\documentclass{amsart}
\usepackage{amscd,amsmath,amssymb,amsfonts}
\usepackage[cmtip, all]{xy}

\newtheorem{thm}{Theorem}[section]
\newtheorem{prop}[thm]{Proposition}
\newtheorem{lem}[thm]{Lemma}
\newtheorem{cor}[thm]{Corollary}
\newtheorem{conj}[thm]{Conjecture}

\renewcommand{\theclaim}{\kern-3pt}

\theoremstyle{definition}
\newtheorem{definition}[thm]{Definition}

\theoremstyle{remark}
\newtheorem{rem}[thm]{Remark}
\newtheorem{rems}[thm]{Remarks}

\numberwithin{equation}{section}

\newcommand{\sC}{{\mathcal C}}

\newcommand{\sE}{{\mathcal E}}

\newcommand{\sH}{{\mathcal H}}

\newcommand{\sM}{{\mathcal M}}

\newcommand{\sO}{{\mathcal O}}

\newcommand{\sS}{{\mathcal S}}

\newcommand{\sZ}{{\mathcal Z}}
\newcommand{\A}{{\mathbb A}}

\newcommand{\N}{{\mathbb N}}
\renewcommand{\P}{{\mathbb P}}

\newcommand{\W}{{\mathbb W}}

\newcommand{\Z}{{\mathbb Z}}

\renewcommand{\L}{{\mathbb L}}

\renewcommand{\phi}{\varphi}

\newcommand{\codim}{{\rm codim}}

\newcommand{\Hom}{{\rm Hom}}

\newcommand{\Spec}{{\rm Spec \,}}

\newcommand{\id}{{\operatorname{id}}}

\newcommand{\Sch}{{\operatorname{\mathbf{Sch}}}}

\newcommand{\op}{{\text{\rm op}}}

\newcommand{\del}{\partial}

\newcommand{\Sm}{{\mathbf{Sm}}}

 \newcommand{\Ab}{{\mathbf{Ab}}}

\newcommand{\SH}{{\operatorname{\sS\sH}}}

\newcommand{\ds}{{/\kern-3pt/}}

\newcommand{\res}{{\operatorname{res}}}

\newcommand{\Gr}{{\mathop{\rm{Grass}}}}

\newcommand{\Mod}{{\operatorname{Mod}}}
\newcommand{\SmOp}{{\mathbf{SmOp}}}
\newcommand{\SP}{{\mathbf{SP}}}
\newcommand{\gr}{\mathbf{Gr}}
\newcommand{\bgr}{\mathbf{bi\text{-}Gr}}
\newcommand{\MGL}{{\operatorname{MGL}}}
\newcommand{\BGL}{{\operatorname{BGL}}}

\begin{document}

\title[Oriented cohomology]{Oriented cohomology, Borel-Moore homology and algebraic cobordism}
\author{Marc Levine}
\address{
Department of Mathematics\\
Northeastern University\\
Boston, MA 02115\\
USA}
\email{marc@neu.edu}


 
 \thanks{The
author thanks the NSF for support via grant DMS-0457195}

\renewcommand{\abstractname}{Abstract}
\begin{abstract}   We examine various versions of oriented cohomology and Borel-Moore homology theories in algebraic geometry and put these two together in the setting of an ``oriented duality theory", a generalization of Bloch-Ogus twisted duality theory. We apply this to give a Borel-Moore homology version $\MGL'_{*,*}$  of Voevodsky's  $\MGL^{*,*}$-theory, and a natural map $\vartheta:\Omega_*\to \MGL'_{2*,*}$, where $\Omega_*$ is the algebraic cobordism theory defined in \cite{LevineMorel}. We conjecture that $\vartheta$ is an isomorphism and describe a program for proving this conjecture.
\end{abstract}
\date{December 31, 2007}
\maketitle
\tableofcontents

\section*{Introduction}  The notion of {\em oriented cohomology} has been introduced, in various forms and in various settings, in the work of Panin \cite{Panin2}, Levine-Morel \cite{LevineMorel},   and others. A related notion, that of {\em oriented Borel-Moore homology} appears in \cite{LevineMorel}.  Mocanasu \cite{Mocanasu} has examined the relation of these two notions, and, with a somewhat different axiomatic as appearing in either  \cite{LevineMorel} or \cite{Panin2}, has given an equivalence of these two theories, the relation being that the cohomology with supports in a closed subset $X$ of a smooth scheme $M$ becomes the Borel-Moore homology of $X$. 

Our main goal in this paper is to tie all these theories together. Our first step is to extend results of \cite{Panin2}, to show that an orientation on a ring cohomology theory gives rise to a good theory of projective push-forwards on the cohomology with supports. This extension of Panin's results allows us to use the ideas and results of Mocanasu, which in essence show that many of the properties and structures associated with the cohomology of a smooth scheme $M$ with supports in a closed subset $X$ depend only on $X$; we require resolution of singularities for this step. We axiomatize this into the notion of an {\em oriented duality theory}, which one can view as a version of the classical notion of a Bloch-Ogus twisted duality theory. The main difference between a general oriented duality theory $(H,A)$ and a Bloch-Ogus theory is that one does not assume that the Chern class map $L\mapsto c_1(L)$ satisfies the usual additivity with respect to tensor product of line bundles:
\[
c_1(L\otimes M)=c_1(L)+c_1(M).
\]
This relation gets replaced with the {\em formal group law} $F_A(u,v)\in A(\Spec k)[[u,v]]$ of the underlying oriented cohomology theory $A$, defined by the relation
\[
c_1(L\otimes M)=F_A(c_1(L),c_1(M)).
\]
In fact, the Chern classes $c_1(L)$ and formal group law $F_A$ are not explicitly given as part of the axioms, but follows from the more basic structures, namely, the pull-back,  projective push-forward and the projective bundle formula.

We conclude with a discussion of the two theories which form our primary interest: the theory of algebraic cobordism $\Omega_*$ of \cite{LevineMorel}, and the bi-graded theory $\MGL^{*,*}$, also known as algebraic cobordism, but defined via the algebraic Thom complex $\MGL$ in the Morel-Voevodsky motivic stable homotopy category $\SH(k)$ (see \cite{Voevodsky}). Assuming that $k$ admits resolution of singularities, we show how one may apply our general theory to $\MGL^{*,*}$, giving rise to an associated oriented Borel-Moore homology theory $\MGL'_{*,*}$, which together form an oriented duality theory $(\MGL'_{*,*}, \MGL^{*,*})$. Concerning $\Omega_*$, we show how this theory comes with a canonical ``classifying map"
\[
\vartheta_H:\Omega_*\to H_{2*,*}
\]
for each bi-graded oriented duality theory $(H,A)$. Taking the case $(\MGL'_{*,*}, \MGL^{*,*})$, we achieve an extension of the natural transformation $\vartheta^\MGL:\Omega^*\to \MGL^{2*,*}$, discussed in \cite{LevineMorel}, to a natural transformation
\[
\vartheta_\MGL:\Omega_*\to \MGL'_{2*,*}.
\]
We conjecture that $\vartheta$ is an isomorphism, extending the conjecture of  \cite{LevineMorel} that $\vartheta^\MGL$ is an isomorphism, and we outline a program for proving this conjecture. In fact, the extension of the conjecture of \cite{LevineMorel}, and how this extension to the setting of Borel-Moore homology could lead to a proof of the conjecture, is the main motivation behind this paper.

In the first section, we review Panin's theory of oriented ring cohomology and show how his method of defining projective push-forwards for oriented ring cohomology extends to give projective push-forwards for cohomology with supports. In section~\ref{sec:Mocanasu}, we recall Mocanasu's theory of algebraic oriented cohomology, giving a modified version of this theory, and  show that the projective push-forward with supports defined in section~\ref{sec:Integration} endows an oriented ring cohomology theory  with the structure of an algebraic oriented cohomology theory. In section~\ref{sec:OrientDualityThy} we introduce the notion of an oriented duality theory and show that an oriented ring cohomology theory extends uniquely to  an oriented duality theory. In the last section, we apply our results to $\MGL^{*,*}$, construct the classifying map $\vartheta_H:\Omega_*\to H_{2*,*}$, and discuss the conjecture that $\vartheta_\MGL$ is an isomorphism.

\section{Integration with support}\label{sec:Integration}  Panin has made a study of properties of {\em oriented ring cohomology theories}, showing how a good theory of Chern classes of line bundles gives rise to push-forward maps for projective morphisms (he calls this latter structure an {\em integration}). For our purposes, we will need push-forward maps for projective morphisms of pairs, so we need to extend Panin's theory a bit. Fortunately, the extension is mainly a matter of making a few changes in the definitions, and noting that most of Panin's arguments extend without major change to the more general setting. In this section, we give the necessary extension of Panin's theory.

We fix a base-field $k$ and let $\Sm/k$ denote the category of smooth, quasi-projective varieties over $k$. We denote the base-scheme $\Spec k$ by $pt$. Panin uses the category $\SmOp$ of {\em smooth open pairs} over $k$, this being the category of pairs $(M,U)$, $M,U\in\Sm/k$, with $U\subset M$ an open subscheme. A morphism $f:(M,U)\to (N,V)$ is a morphism $f:M\to N$ in $\Sm/k$ with $f(U)\subset V$. 

For $(M,U)\in\SmOp$ and for $X=M\setminus U$, we call the pair $(M,X)$ a {\em  smooth pair}. If we define a morphism of smooth pairs, $f:(N,Y)\to (M,X)$ to be a morphism $f:N\to M$ in $\Sm/k$ such that $f^{-1}(X)\subset Y$, then we have the evident isomorphism of  $\SmOp$ with the category $\SP$ of smooth pairs. We will use throughout $\SP$ instead of $\SmOp$. We have the inclusion functor $\iota:\Sm/k\to \SP$ sending $M$ to $(M,M)$ and $f:M\to N$ to the induced map $f:(M,M)\to (N,N)$. 

As our intention here is to add support conditions to Panin's theory, and this requires some additional commutativity conditions not imposed in \cite{Panin2}, we will add the simplifying assumption that a ring cohomology theory will always be  $\Z/2$-graded. We likewise require that the boundary maps in the underlying cohomology theory are of odd degree and that the pull-back maps preserve degree.  Having made these modifications, we have the following version of Panin's notion of a cohomology theory, and a ring cohomology theory, on $\SP$.

\begin{definition}\label{def:cohomology}
A {\em cohomology theory} $A$ on $\SP$ is  a   functor $A:\SP^\op\to \gr_{\Z/2}\Ab$, 
together with a collection of degree 1 operators 
\[
\del_{M,X}:A(M,X)\to A(M\setminus X,M\setminus X)
\]
 satisfying the axioms of \cite[Definition 2.0.1]{Panin2}. 
 
 For a smooth pair $(M,X)$, we write $A_X(M)$ for $A(M,X)$, we write $A(M)$ for $A(M,M)=A_M(M)$, and for $f:(M,X)\to (N,Y)$ a morphism in $\SP$, we write $f^*:A_Y(N)\to A_X(M)$ for the map $A(f)$. For a smooth pair $(M,X)$,  the identity map on $M$ induces the ``forget the support map" $\id_M^*:A_X(M)\to A(M)$.   With these notations, the axioms are:
 \begin{enumerate}
 \item {\em localization}: For each $(M,X)\in\SP$,  let $U=M\setminus X$ and let $j:U\to M$ be the inclusion. Then the sequence
 \[
 A(M)\xrightarrow{j^*}A(U)\xrightarrow{\del_{M,X}}A_X(M)\xrightarrow{\id_M^*}A(M)\xrightarrow{j^*}
 A(U)
 \]
 is exact.
 
 In addition, the maps $\del_{M,X}$ are natural with respect to morphisms in $\SP$: given a morphism $f:(M,X)\to (N,Y)$ in $\SP$, the diagram
 \[
 \xymatrix{
 A(N\setminus Y)\ar[r]^{\del_{N,Y}}\ar[d]_{f_{|N\setminus Y}^*}&A_Y(N)\ar[d]^{f^*}\\
 A(M\setminus X)\ar[r]_{\del_{M,X}}&A_X(M)
 }
 \]
 commutes.
 \item {\em excision}: Let $f:M'\to M$ be an \'etale morphism in $\Sm/k$, let $X\subset M$ be a closed subset and suppose that $f:f^{-1}(X)\to X$ is an isomorphism (giving $X$ and $f^{-1}(X)$ the reduced scheme structure). Then the map $f:(M',f^{-1}(X))\to (M,X)$ induces an isomorphism $f^*:A_X(M)\to A_{f^{-1}(X)}(M')$.
 \item {\em homotopy}: For $M\in\Sm/k$, the map $p^*:A(M)\to A(M\times\A^1)$ induced by the projection $p:M\times\A^1\to M$ is an isomorphism.
 \end{enumerate}
 \end{definition}

\begin{rems}\label{rems:Cohomology} 1. The localization and excision axioms yield a long exact Mayer-Vietoris sequence. Similarly, if $X\subset X'\subset M$ are closed subsets of $M\in\Sm/k$, putting together the localization sequences for $X\subset M$, $X'\subset M$ and $X'\setminus X\subset M\setminus X$ gives the exact sequence of the triple $(M,X',X)$:
\begin{multline*}
 A_{X'}(M)\xrightarrow{j^*}A_{X'\setminus X}(M\setminus X)\\\xrightarrow{\del_{M,X',X}}A_X(M)\xrightarrow{\id_M^*}A_{X'}(M)\xrightarrow{j^*}
 A_{X'\setminus X}(M\setminus X).
\end{multline*}
See \cite[2.2.3]{Panin2} for details.\\\\
2. Let $p:V\to M$ be an affine space bundle, $X\subset M$ a closed subset. Together with localization and Mayer-Vietoris, the homotopy axiom implies that
\[
p^*:A_X(M)\to A_{p^{-1}(X)}(V)
\]
is an isomorphism. Indeed, $V\to M$ is Zariski locally isomorphic to the projection $M\times\A^n\to M$.
\end{rems}
\begin{definition}
A {\em ring cohomology theory} on $\SP$ is a cohomology theory $A$ on $\SP$ together with graded maps for each pair of   smooth pairs $(M,X)$, $(N,Y)$ 
\[
\times:A_X(M)\otimes A_Y(N)\to A_{X\times Y}(M\times N)
\]
and an element $1\in A^{ev}(pt)$
satisfying the axioms of \cite[Definition 2.4.2]{Panin2}:
\begin{enumerate}
\item {\em associativity}: $(a\times b)\times c=a\times(b\times c)$.
\item {\em unit}:  $a\times1=1\times a=a$.  
\item{\em the partial Leibniz rule}:  Given smooth pairs $(M,X)$, $(M,X')$, $(N,Y)$, with $X\subset X'$, we have the exact sequence of the triple $(M\times N, X'\times Y, X\times Y)$ (remark~\ref{rems:Cohomology}(1)) with boundary map
\[
\del_{M\times N, X'\times N, X\times N}:A_{(X'\setminus X)\times Y}((M\setminus Z)\times N)\to A_{X\times Y}(M\times N).
\]
We  also have the triple $(M,X',X)$, with boundary map
\[
\del_{M,X',X}:A_{X'\setminus X}(M\setminus X)\to A_X(M).
\]
Then  
\[
\del_{M\times N,X'\times N,X\times N}(a\times b)=\del_{M,X',X}(a)\times b
\]
for $a\in A_{X'\setminus X}(M\setminus X)$, $b\in A_Y(N)$.  
\end{enumerate}
We add the axiom 
\begin{enumerate}
\item[(4)] {\em graded commutativity}: For $a\in A_X(M)$ of degree $p$ and $b\in A_Y(N)$ of degree $q$, and with $\tau:N\times M\to M\times N$ denoting the symmetry isomorphism, we have
\[
\tau^*(a\times b)=(-1)^{pq}(b\times a).
\]
\end{enumerate}
\end{definition}

For smooth pairs $(M,X), (M,Y)$, pull-back by the diagonal gives the {\em cup product with supports}
\[
\cup:A_X(M)\otimes A_Y(M)\to A_{X\cap Y}(M);\quad \cup:=\delta_X^*\circ \times.
\]
In particular, $A(M)$ is a $\Z/2$-graded, graded-commutative  ring with unit for each $M\in\Sm/k$ and $A_X(M)$ is an $A(M)$-module for each smooth pair $(M,X)$; $A_X(M)$ is itself a $\Z/2$-graded, graded-commutative  ring without unit. The partial Leibniz rule implies that, for a triple $(M,X',X)$, the boundary map 
\[
\del_{M,X',X}:A_{X'\setminus X}(M\setminus X)\to A_X(M)
\]
is  an $A_{X''}(M)$-module map for all closed subsets $X'\subset X''\subset M$; more generally, for any closed $X''\subset M$, we have
\[
\del_{M,X'\cap X'', X\cap X''}(a\cup b)=a\cup \del_{M,X',X}(b)\in A_{X\cap X''}(M)
\]
for $a\in A_{X''}(M)$, $b\in A_{X'\setminus X}(M\setminus X)$. 

\begin{rem} Instead of a $\Z/2$ grading, one can work in the $\Z$-graded or bi-graded setting. One requires that the pull-back maps $f^*$ preserve the (bi-)grading, and that $\del$ is of degree $+1$ or bi-degree $(+1,0)$.
\end{rem}

 At various places in the theory, Panin requires various elements to have certain commutativity properties (see e.g., \cite[Definition 2.4.4]{Panin2}); we will replace these conditions with the condition that these elements have even degree.  With these modifications, Panin defines four structures on a ring cohomology theory $A$:
\begin{enumerate}
\item An {\em orientation} on $A$ is an assignment of a graded $A(M)$-module isomorphism $th^E_X:A_X(M)\to A_X(E)$ for each smooth pair $(M,X)$ and each vector bundle $E$ on $M$, satisfying the properties listed in \cite[Definition 3.1.1]{Panin2}.
\item A {\em Chern structure} on $A$ is an assignment of an even degree element $c_1(L)\in A^{ev}(M)$ for each line bundle $L$ on $M\in\Sm/k$, satisfying the properties of functoriality, nondegeneracy: a $\P^1$-bundle formula,  and vanishing for $L=O_M$ the trivial line bundle on $M$ (see \cite[Definition 3.2.1]{Panin2}).
\item A {\em Thom structure} on $A$ is the assignment of an even degree element $th(L)\in A^{ev}_M(L)$ for each line bundle $L$ on $M\in\Sm/k$, satisfying the properties of functoriality and nondegeneracy: cup product with $th(O_M)$ is an isomorphism $\cup th(O_M):A(M)\to A_M(M\times\A^1)$ (see \cite[Definition 3.2.2]{Panin2}).
\item An {\em integration} on $A$ is a assignment $(f:N\to M)\mapsto f_*:A(N)\to A(M)$ for each projective morphism $f:N\to M$ in $\Sm/k$, satisfying the properties of \cite[Definition 4.1.2]{Panin2} (since we are in the $\Z/2$-graded setting, we require that $f_*$ preserves the grading).
\end{enumerate}
The main result of \cite{Panin2} is that each of these structures gives rise in a uniquely determined manner to all the other structures, and that each ``loop" in this process induces the identity transformation. Our goal in this section is to extend this result to a more widely defined integration structure. 

 \begin{rem}\label{rem:Grading1} In the $\Z$-graded or bi-graded situation one  requires that the Chern class $c_1(L)$ or the Thom class $th(L)$ is in degree 2 (in the graded case) or bi-degree $(2,1)$ (in the bi-graded case), and that the push-forward $f_*$ shifts (bi-)degrees
 \[
 f_*:A^n(N)\to A^{n+2d}(M);\quad f_*:A^{p,q}(N)\to A^{p+2d,q+d}(M)
 \]
 where $d$ is the codimension of $f$, $d=\dim_kM-\dim_kN$. With these  modifications, one recovers Panin's main results in the (bi-)graded case.
 \end{rem}

Let $\SP'$ be the category with objects the smooth pairs $(M,X)$, $M\in \Sm/k$, $X\subset M$, where a morphism $f:(M,X)\to (N,Y)$ is a projective morphisms of pairs, i.e., a projective morphism $f:M\to N$ in $\Sm/k$ such that $f(X)\subset Y$. A  morphism $f:N\to M$ in $\Sm/k$ and  closed subsets $Z\subset M$ and $Y\subset N$ give rise to  the map
\[
f^*(-)\cup:A_Z(M)\otimes A_Y(N)\to A_{Y\cap f^{-1}(Z)}(N)
\]
sending $a\otimes b$ to $f^*(a)\cup b$, where $f^*:A_Z(M)\to A_{f^{-1}(Z)}(N)$ is the pull-back. For $Y\subset f^{-1}(Z)$, the map $f^*(-)\cup$ makes $A_Y(N)$ an $A_Z(M)$-module.

\begin{definition}\label{def:IntSupp} Let $A$ be a $\Z/2$-graded ring cohomology theory on $\SP$. An {\em integration with supports} on $A$ is an assignment of a graded push-forward map
\[
f_*:A_Y(N)\to A_X(M)
\]
for each   morphism $f:(N,Y)\to (M,X)$ in $\SP'$, satisfying:
\begin{enumerate}
\item $(f\circ g)_*=f_*\circ g_*$ for composable   morphisms.
\item For $f:(N,Y)\to (M,X)$ in $\SP'$, and $Z$ a closed subset of $M$,  $f_*$ is a $A_Z(M)$-module map, i.e., the diagram
\[
\xymatrixcolsep{40pt}
\xymatrix{
 A_Z(M)\otimes A_Y(N)\ar[r]^-{f^*(-)\cup} \ar[d]_{\id\otimes f_*}&A_{Y\cap f^{-1}(Z)}(N)\ar[d]^{f_*}\\
 A_Z(M)\otimes A_X(M)\ar[r]_-{\cup}&A_{X\cap Z}(M)
}
\]
commutes.
\item Let $i:(N,Y)\to (M,X)$ be morphism in $\SP'$ such that $i:N\to M$ is a closed embedding, let  $g:(\tilde{M},\tilde{X})\to (M,X)$ be a morphism in $\SP$. Let $\tilde{N}:=N\times_M\tilde{M}$, $\tilde{g}:\tilde{N}\to N$, $\tilde{i}:\tilde{N}\to \tilde{M}$ be the projections, and let $\tilde{Y}:=\tilde{i}^{-1}(Y)$. Suppose in addition that $\tilde{N}$ is in $\Sm/k$ and the square
\[
\xymatrix{
\tilde{N}\ar[d]_{\tilde{g}}\ar[r]^{\tilde{i}}&\tilde{M}\ar[d]^g\\
N\ar[r]_i&M}
\]
is transverse. Then the diagram
\[
\xymatrix{
A_{\tilde{Y}}(\tilde{N})\ar[r]^{\tilde{i}_*}&A_{\tilde{X}}(\tilde{M})\\
A_Y(N)\ar[u]^{\tilde{g}^*}\ar[r]_{i_*}&A_X(M)\ar[u]_{g^*}}
\]
commutes.
\item Let $f:(N,Y)\to (M,X)$ be a morphism in $\SP$, and let $p_N:\P^n\times N\to N$, $p_M:\P^n\times M\to M$ be the projections. Then the diagram
\[
\xymatrix{
A_{\P^n\times Y}(\P^n\times N)\ar[d]_{p_{N*}}&A_{\P^n\times X}(\P^n\times M)\ar[l]_{(\id\times f)^*}\ar[d]^{p_{M*}}\\
A_Y(N)&A_X(M)\ar[l]^{f^*}}
\]
commutes.
\item Given smooth pairs  $(M,X)$ and $(M,Y)$ with   $X\subset Y$, the maps
\[
\id_{M*}:A_X(M)\to A_Y(M)
\]
and
\[
\id_M^*:A_X(M)\to A_Y(M)
\]
are equal.
\item  Let $f:N\to M$ be a projective morphism in $\Sm/k$, let $Y\subset Y'\subset N$, $X\subset X'\subset M$ be closed subsets, and suppose that $f^{-1}(X)\cap Y'=Y$, $f(Y')\subset X'$. Then the diagram
\[
\xymatrixcolsep{40pt}
\xymatrix{
A_{Y'\setminus Y}(N\setminus Y)\ar[r]^-{\del_{N,Y',Y}}\ar[d]_{f_*}&A_Y(N)\ar[d]^{f_*}\\
A_{X'\setminus X}(M\setminus X)\ar[r]_-{\del_{M,X',X}}&A_X(M)}
\]
commutes. Here $\del_{N,Y',Y}$ and $\del_{M,X',X}$ are the boundary maps in the respective long exact sequence for the triples $(N,Y',Y)$ and $(M,X',X)$, and the push-forward map 
$f_*:A_{Y'\setminus Y}(N\setminus Y)\to A_{X'\setminus X}(M)$ is the composition
\[
A_{Y'\setminus Y}(N\setminus Y)\xrightarrow{j^*}A_{Y'\setminus Y}(N\setminus f^{-1}(X))
\xrightarrow{f_{|N\setminus f^{-1}(X)*}}A_{X'\setminus X}(M\setminus X).
\]
\end{enumerate}
\end{definition}
Note that an integration with supports on $A$  determines an integration on $A$ by restricting $f_*$ to $f_*:A(N)\to A(M)$. One has as well a $\Z$-graded or bi-graded version.

 For later use, we give an extension of the properties (3) and (4) of definition~\ref{def:IntSupp}.

\begin{lem} \label{lem:PushPull} Let $A$ be a $\Z/2$-graded ring cohomology theory on $\SP$, with an integration with supports. Let $f:(N,Y)\to (M,X)$ be a   morphism in $\SP'$, $g:(\tilde{M},\tilde{X})\to (M,X)$ a morphism in $\SP$. Let $\tilde{N}:=N\times_M\tilde{M}$, $\tilde{g}:\tilde{N}\to N$, $\tilde{f}:\tilde{N}\to \tilde{M}$ be the projections, and let $\tilde{Y}:=\tilde{f}^{-1}(X)$. Suppose in addition that there is an open neighborhood $U$ of $Y$ in $N$ such that $U\times_M\tilde{M}$ is in $\Sm/k$, the diagram
\[
\xymatrix{
U\times_M\tilde{M}\ar[d]_{\tilde{g}}\ar[r]^{\tilde{f}}&\tilde{M}\ar[d]^g\\
U\ar[r]_f&M}
\]
is transverse, and  the closure $\hat{U}$ of $U\times_M\tilde{M}$ in $\tilde{N}$ is smooth. Then the diagram
\[
\xymatrix{
A_{\tilde{Y}}(\hat{U})\ar[r]^{\hat{f}_*}&A_{\tilde{X}}(\tilde{M})\\
A_Y(N)\ar[u]^{\hat{g}^*}\ar[r]_{f_*}&A_X(M)\ar[u]_{g^*}}
\]
commutes, where $\hat{f}:\hat{U}\to \tilde{M}$ is the restriction of $\tilde{f}$ and $\hat{g}:\hat{U}\to N$ is the restriction of $\tilde{g}$.

In particular, if $\tilde{N}$ is in $\Sm/k$ and the cartesian diagram
\[
\xymatrix{
\tilde{N}\ar[d]_{\tilde{g}}\ar[r]^{\tilde{f}}&\tilde{M}\ar[d]^g\\
N\ar[r]_f&M}
\]
is transverse, then the diagram
\[
\xymatrix{
A_{\tilde{Y}}(\tilde{N})\ar[r]^{\tilde{f}_*}&A_{\tilde{X}}(\tilde{M})\\
A_Y(N)\ar[u]^{\hat{g}^*}\ar[r]_{f_*}&A_X(M)\ar[u]_{g^*}}
\]
commutes.
\end{lem}

\begin{proof} Factor $f:N\to M$ as $p\circ i$, with $i:N\to \P^n\times M$ a closed immersion, and $p:\P^n\to M$ the projection. The statement for $f=p$ is just definition~\ref{def:IntSupp}(3), so we need only handle the case of $f=i$ a closed immersion. Also, it suffices to handle the case $\tilde{X}=g^{-1}(X)$.

Set  $F:=N\setminus U$, and let $V:=M\setminus F$, $\tilde{V}:=g^{-1}(V)$. Then $V$ is a neighborhood of $X$ in $M$, and $g^{-1}(V)$ is a neighborhood of $\tilde{X}=g^{-1}(X)$ in $\tilde{M}$. Letting $\tilde{U}=U\times_M\tilde{M}$, we have the transverse cartesian diagram in $\Sm/k$
\[
\xymatrix{
\tilde{U}\ar[d]_{\tilde{g}_U}\ar[r]^{\tilde{i}_U}&\tilde{V}\ar[d]^g_U\\
U\ar[r]_{i_U}&V}
\]
with $i_U$ a closed immersion. Definition~\ref{def:IntSupp}(2) gives us the commutative diagram
\[
\xymatrix{
A_{\tilde{Y}}(\tilde{U})\ar[r]^{\tilde{i}_{U*}}&A_{\tilde{X}}(\tilde{V})\\
A_Y(U)\ar[u]^{\hat{g}_U^*}\ar[r]_{i_{U*}}&A_X(V)\ar[u]_{g_U^*}}
\]
Now we just use the excision isomorphisms
\[
A_{\tilde{Y}}(\hat{U})\to A_{\tilde{Y}}(\tilde{U}),\
A_Y(N)\to A_Y(U),\ A_{\tilde{X}}(\tilde{M})\to A_{\tilde{X}}(\tilde{V}),\
A_X(M)\to A_X(V)
\]
and definition~\ref{def:IntSupp}(2) to give the commutativity of 
\[
\xymatrix{
A_{\tilde{Y}}(\hat{U})\ar[r]^{\hat{i}_*}&A_{\tilde{X}}(\tilde{M})\\
A_Y(N)\ar[u]^{\hat{g}^*}\ar[r]_{i_*}&A_X(M)\ar[u]_{g^*}}
\]
\end{proof}

 \begin{rem}[Projection formula] \label{rem:ProjForm} The condition (2) of definition~\ref{def:IntSupp} is just the projection formula ``with supports", i.e., that
 \[
 f_*(f^*(a)\cup b)=a\cup f_*(b)\in A_{X\cap Z}(M)
 \]
 for $a\in A_Z(M)$, $b\in A_Y(N)$ and $f:(N,Y)\to (M,X)$ a morphism in $\SP'$. This axiom may also be stated using the external product instead of the cup product: Let $f:(N,Y)\to (M,X)$, $g:(N',Y')\to (M',X')$ be morphisms in $\SP'$. Then
 \begin{equation}\label{eqn:ExtProd}
 (f\times g)_*(a\times b)=f_*(a)\times g_*(b)\in A_{X\times X'}(M\times M'),
 \end{equation}
 for all $a\in A_Y(N)$, $b\in A_{Y'}(N')$. Indeed, to recover (2), take $g$ to be $\id:(M,Z)\to (M,Z)$, and apply (3) to the transverse cartesian diagram
 \[
 \xymatrix{
 N\ar[r]^{(\id,f)}\ar[d]_f&N\times M\ar[d]^{f\times\id}\\
 M\ar[r]_\delta&M\times M
 }
 \]
 and morphism $\delta:(M,X\cap Z)\to (M\times M, X\times Z)$. 
 
 To see that (2) implies \eqref{eqn:ExtProd}, since $f\times g=(f\times\id)\circ (\id\times g)$, it suffices to handle the case $g=\id$. From the commutative diagram
 \[
 \xymatrix{
 N\times M'\ar[r]^{p'_2}\ar[d]_{f\times\id}&M'\ar@{=}[d]\\
 M\times M'\ar[r]_{p_2}&M'
 }
 \]
 we have $p_2^{\prime*}(b)=(f\times\id)^*(p_2^*(b))$.  Applying (3) to the cartesian transverse diagram
 \[
 \xymatrix{
 N\times M'\ar[r]^{p_1'}\ar[d]_{f\times\id}&N\ar[d]^f\\
 M\times M'\ar[r]_{p_1}&M
 }
 \]
 and using (2) gives
 \begin{align*}
 (f\times\id)_*(a\times b)&=(f\times\id)_*(p_1^{\prime*}(a)\cup p_2^{\prime*}(b))\\
 &=(f\times\id)_*(p_1^{\prime*}(a)\cup(f\times\id)^*(p_2^*(b)))\\
 &=(f\times\id)_*(p_1^{\prime*}(a))\cup p_2^*(b)\\
 &=p_1^*(f_*(a))\cup p_2^*(b)\\
 &=f_*(a)\times b
 \end{align*}
 \end{rem}

To set up a one-to-one correspondence between integrations with support and the other structures, we rephase the compatibility condition \cite[Definition 4.1.3]{Panin2}.

\begin{definition}\label{def:Comp} Let $\omega$ be an orientation of $A$ and $L\mapsto c_1(L)$ the corresponding Chern structure on $A$ (given by \cite[3.7.5]{Panin2}). We say that an integration with supports $f\mapsto f_*$ on $A$ is subjected to the orientation $\omega$ if for each smooth pair $(M,X)$ and each line bundle $p:L\to M$ with zero-section $s:M\to L$, the endomorphism
\[
A_X(M)\xrightarrow{s_*}A_{p^{-1}(X)}(L)\xrightarrow{s^*}A_X(M)
\]
of $A_X(M)$ is given by cup product with $c_1(L)$. 
\end{definition}
In case $X=M$, this condition is just saying that $c_1(L)=s^*(s_*(1))$, which the reader will easily check is equivalent to the condition given in Panin's definition  \cite[Definition 4.1.3]{Panin2}.

Our main result is (compare with \cite[Theorem 4.1.4]{Panin2}):

\begin{thm}\label{thm:Int} Let $A$ be a $\Z/2$-graded ring cohomology theory. Given an  orientation $\omega$ on $A$ there is a unique integration with supports on $A$ subjected to $\omega$. 
 \end{thm}
 
\begin{cor} Let $A$ be a $\Z/2$-graded ring cohomology theory. Given an integration $f\mapsto f_*$ on $A$ there is a unique integration with supports on $A$ extending $f$.
\end{cor}

\begin{proof} Let $\omega$ be the orientation on $A$ corresponding to $f$ by \cite[Theorem 4.1.4]{Panin2}. By theorem~\ref{thm:Int}, there is a unique  integration with supports $\iota$ on $A$ subjected to $\omega$. Since the restriction of $\iota$ to an integration on $A$ (without supports) is subjected to $\omega$, it follows from the uniqueness in \cite[Theorem 4.1.4]{Panin2} that the restriction of $\iota$ to an integration on $A$ is the given one  $f\mapsto f_*$.  Thus an extension of  $f\mapsto f_*$ to an integration with supports on $A$ exists.

If now $\iota'$ is another extension, write the push-forward map for $f$ as $f'_*$; note that $f'_*=f_*$ if we omit supports. Take $a\in A_X(M)$ and let $p:L\to M$ be a line bundle with zero section $s$. Then
\begin{align*}
s^*(s'_*(a))&=s^*(s'_*(s^*p^*(a)\cup 1))\\
&=s^*(p^*(a)\cup s_*(1))\\
&=a\cup s^*(s_*(1))\\
&=a\cup c_1(L).
\end{align*}
Thus $\iota'$ is subjected to $\omega$, and hence $\iota=\iota'$ by the uniqueness in theorem~\ref{thm:Int}.
\end{proof}

Theorem~\ref{thm:Int} is proven by copying the construction in \cite{Panin2} of an integration subjected to a given orientation $\omega$, making at each stage the extension to an integration with supports. \\
\\
{\bf Step 1}: {\em The case of a closed immersion}. Let $i:N\to M$ be a closed immersion in $\Sm/k$,  $Y\subset N$ a closed subset and let $\nu\to N$ be the normal bundle of $N$ in $M$. The deformation to the normal bundle (\cite[\S 2.2.7]{Panin2}) gives the diagram
\[
\xymatrix{
\nu\ar[r]^{i_0}&M_t&M\ar[l]_{i_1}\\
N\ar[r]_{i_0}\ar[u]^s&N\times\A^1\ar[u]_{\tilde{i}}&N\ar[l]^{i_1}\ar[u]_i\\
Y\ar@{^(->}[u]\ar[r]_{i_0}&Y\times\A^1\ar@{^(->}[u]&Y.\ar[l]^{i_1}\ar@{^(->}[u]}
\]
From this deformation diagram,  we arrive at the maps
\[
A_Y(\nu)\xleftarrow{i_0^*}A_{Y\times\A^1}(M_t)\xrightarrow{i_1^*}A_Y(M).
\]
\begin{lem}\label{lem:Htpy} The maps $i_0^*, i_1^*$ are isomorphisms.
\end{lem}

\begin{proof} In case $Y=N$, this is \cite[Theorem 2.2.8]{Panin2}. In general, let $U=M\setminus i(Y)$, $V=N\setminus Y$, let $\nu'$ be the normal bundle of $V$ in $U$, and let $U_t$ be the deformation space constructed from the closed immersion $i':V\to U$. 

Let $j:U_t\to M_t\setminus Y\times\A^1$, $\bar{j}:\nu'\to \nu\setminus Y$ be the inclusions. We have the commutative diagram
\[
\xymatrix{
A_V(\nu\setminus Y)\ar[d]_{\bar{j}^*}&\ar[l]_-{i_0^*}A_{V\times\A^1}(M_t\setminus Y\times\A^1)\ar[d]_{j^*}\ar[r]^-{i_1^*}&A_V(U)\ar@{=}[d]\\
A_V(\nu')&\ar[l]^-{i_0^*}A_{V\times\A^1}(U_t)\ar[r]_-{i_1^*}&A_V(U);}
\]
the maps $j^*$ and $\bar{j}^*$ are isomorphisms by excision. By \cite[Theorem 2.2.8]{Panin2}, the horizontal maps in the bottom row are isomorphisms, hence the horizontal maps in the top row are isomorphisms as well.

We have the commutative diagram
\[
\xymatrix{
\ar[d]&\ar[d]&\ar[d]\\
A_V(\nu\setminus Y)\ar[d]_\del&A_{V\times\A^1}(M_t\setminus Y\times\A^1)\ar[l]_-{i_0^*}\ar[r]^-{i_1^*}\ar[d]^\del
&A_V(U)\ar[d]^\del\\
A_Y(\nu)\ar[d]&A_{Y\times\A^1}(M_t)\ar[l]_{i_0^*}\ar[d]\ar[r]^{i_1^*}
&A_Y(M)\ar[d]\\
A_N(\nu)\ar[d]&A_{N\times\A^1}(M_t)\ar[l]^-{i_0^*}\ar[r]_-{i_1^*}\ar[d]&A_N(M)\ar[d]\\
 & &
 }
\]
where the columns are the long exact sequences of triples $(\nu,N,Y)$, $(M_t,N\times\A^1, Y\times\A^1)$ and $(M,N,Y)$. Thus, the case $Y=N$, our remarks above, and the five-lemma shows that the horizontal maps in the middle row are isomorphisms, as desired.
\end{proof}

Now let $X\subset M$ be a closed subset containing $i(Y)$. We have the diagram
\[
A_Y(\nu)\xleftarrow{i_0^*}A_{Y\times\A^1}(M_t)\xrightarrow{i_1^*}A_X(M)
\]
with $i_0^*$ an isomorphism.  Let
\[
i_*:A_Y(N)\to A_X(M)
\]
be given by the composition
\begin{equation}\label{eqn:Gysin}
A_Y(N)\xrightarrow{th^{\nu}_Y}A_Y(\nu)\xrightarrow{i_1^*\circ(i_0^*)^{-1}}A_X(M).
\end{equation}
 
\begin{prop}\label{prop:Gysin} Let $A$ be a $\Z/2$-graded oriented ring cohomology theory.
\begin{enumerate}
\item For $i:N\to M$ a closed immersion in $\Sm/k$, the map $i_*:A(N)\to A(M)$ defined above agrees with the map $i_{gys}$ defined in \cite[\S 4.2]{Panin2}.
\item  Let $i:N\to M$ be a closed immersion in $\Sm/k$, $Y$ a  closed subset of $N$, $X$ a closed subset of $M$ such that $i(Y)\subset X$. Then for $Z\subset M$ a closed subset,  $i_*:A_Y(N)\to A_X(M)$ is an $A_Z(M)$-module homomorphism   (in the sense of definition~\ref{def:IntSupp}(2)).
\item Let $i_1:N\to M$, $i_2:P\to N$ be closed immersions in $\Sm/k$, $X\subset M$, $Y\subset N$ and $Z\subset P$ closed subsets with $i_1(X)\subset Y$, $i_2(Y)\subset Z$. Then
\[
(i_1\circ i_2)_*=i_{1*}\circ i_{2*}:A_Z(P)\to A_X(M).
\]
\item Let $N_1, N_2$ be in $\Sm/k$, and let $j_i:N_i\to N:=N_1\amalg N_2$ be the canonical inclusions, $i=1,2$. Let $i:N\to M$ be a closed immersion in $\Sm/k$, let $Y_i\subset N_i$ be a closed subset, $i=1,2$, and let $X\subset M$ be a closed subset containing $i(Y_1\amalg Y_2)$. Let $i_j$ be the restriction of $i$ to $N_j$, $j=1,2$. Then 
\[
i_*=i_{1*}\circ j_1^*+i_2\circ j_2^*:A_{Y_1\amalg Y_2}(N_1\amalg N_2)\to A_X(M).
\]
\item Let $i:(N,Y)\to (M,X)$ be a morphism in $\SP'$ such that $i:N\to M$ is a closed immersion, and let  $g:(\tilde{M},\tilde{X})\to (M,X)$ be a morphism in $\SP$. Let $\tilde{N}:=N\times_M\tilde{M}$, $\tilde{g}:\tilde{N}\to N$, $\tilde{i}:\tilde{N}\to \tilde{M}$ be the projections, and let $\tilde{Y}:=\tilde{f}{-1}(Y)$. Suppose in addition that $\tilde{N}$ is in $\Sm/k$ and the square
\[
\xymatrix{
\tilde{N}\ar[d]_{\tilde{g}}\ar[r]^{\tilde{i}}&\tilde{M}\ar[d]^g\\
N\ar[r]_i&M}
\]
is transverse. Then the diagram
\[
\xymatrix{
A_{\tilde{Y}}(\tilde{N})\ar[r]^{\tilde{i}_*}&A_{\tilde{X}}(\tilde{M})\\
A_Y(N)\ar[u]^{\tilde{g}^*}\ar[r]_{i_*}&A_X(M)\ar[u]_{g^*}}
\]
commutes.
\item For $M\in\Sm/k$ with closed subsets $Y\subset X$, we have
\[
\id_{M*}=\id_M^*:A_Y(M)\to A_X(M).
\]
\item  Let $i:(N,Y)\to (M,X)$ be a morphism in $\SP'$ such that $i$ is a closed immersion. Let $M\setminus Y\to M$ be the inclusion. Then the sequence
\[
A_Y(N)\xrightarrow{i_*}A_X(M)\xrightarrow{j^*}A_{X\setminus Y}(M\setminus Y)
\]
is exact.
\item Let $i:N\to M$ be a closed immersion in $\Sm/k$, $Y\subset Y'$ closed subsets of $N$, $X\subset X'$ closed subsets of $M$ such that $i^{-1}(X)\cap Y'=Y$, $i(Y')\subset X'$. Then the diagram
\[
\xymatrixcolsep{40pt}
\xymatrix{
A_{Y'\setminus Y}(N\setminus Y)\ar[r]^-{\del_{N,Y',Y}}\ar[d]_{i_*}&A_Y(N)\ar[d]^{i_*}\\
A_{X'\setminus X}(M)\ar[r]_-{\del_{M,X',X}}&A_X(M)}
\]
commutes.
\end{enumerate}
\end{prop}

\begin{proof} (1) follows from the definitions. The proofs of (2)-(7) are exactly as the proofs given in \cite[\S 4.4]{Panin2} of the analogous statements without support, altered by adding in the supports in the notation.

For (8), we may replace $Y'$ with $Y'\cup i^{-1}(X)$ and $Y$ with $i^{-1}(X)$; changing notation, we may assume that $i^{-1}(X)=Y$. Use (1) and (6) to factor $i_*:A_Y(N)\to A_X(M)$ as the composition
\[
A_Y(N)\xrightarrow{i_*}A_Y(M)\xrightarrow{\id_M^*}A_X(M)
\]
and similarly factor  $i_*:A_{Y'\setminus Y}(N\setminus Y)\to A_{X'\setminus X}(M\setminus X)$ as
\[
A_{Y'\setminus Y}(N\setminus Y)\xrightarrow{i_*}A_{Y'\setminus Y}(M\setminus Y)\xrightarrow{j^*}A_{X'\setminus X}(M\setminus X).
\]
Since   the long exact sequence of a triple is natural with respect to pull-back, this reduces us to the case $X=Y$, $X'=Y'$. 

Panin shows \cite[Lemma 3.7.2]{Panin2} that there is a ``Thom classes theory" on $A$, i.e., for each vector bundle $p:E\to M$, $M\in\Sm/k$, an even-degree element $th(E)\in A_M(E)$, such that the orientation isomorphism $th_X^E:A_X(M)\to A_X(E)$ is given by the composition
\[
A_X(M)\xrightarrow{p^*}A_{p^{-11}(X)}(E)\xrightarrow{\cup th(E)}A_X(E).
\]
The classes $th(E)$ satisfy additional properties (see \cite[Definition 3.7.1]{Panin2}), in particular, for $f:N\to M$, we have $f^*(th(E))=th(f^*E)$.  Let $p:\nu\to N$ be the normal bundle of $N$ in $M$, and let $j:N\setminus Y\to N$ be the inclusion. Since the boundary map in the long exact sequence of a triple $(M,X',X)$ is natural with respect to pull-backs and is an $A_{X'}(M)$-module map,  this shows that the diagram
\[
\xymatrixcolsep{40pt}
\xymatrix{
A_{Y'\setminus Y}(N\setminus Y)\ar[d]_{\del_{N,Y',Y}}\ar[r]^-{th^{j^*\nu}_{Y'\setminus Y}}
&A_{Y'\setminus Y}(j^*\nu)\ar[d]^{\del_{\nu,Y',Y}}\\
A_Y(N)\ar[r]_-{th^\nu_Y}&A_Y(\nu)}
\]
commutes.  Looking at the definition \eqref{eqn:Gysin} of the Gysin map, this commutativity, together with  the naturality of the long exact sequence of a triple with respect to pull-back, finishes the proof of (8).
\end{proof}

\begin{rem}\label{rem:GysinIso} Let $i:N\to M$ be a closed immersion in $\Sm/k$, $Y\subset N$ a closed subset, $X=i(Y)$. Then
\[
i_*:A_Y(N)\to A_X(M)
\]
is an isomorphism. Indeed, 
\[
i_*:=i_1^*\circ(i_0^*)^{-1}\circ th^\nu_Y;
\]
$th^\nu_Y$ is  isomorphism by the definition of an orientation, and $i_0^*, i_1^*$ are isomorphisms by lemma~\ref{lem:Htpy}.
\end{rem}

\noindent
{\bf Step 2}: {\em The case of a projection}. This portion relies on the formal group law associated to an oriented theory. We sketch the main points here, following \cite[\S3.9]{Panin2}. 

We recall that an oriented theory $A$ satisfies the {\em projective bundle formula} \cite[Theorem 3.3.1]{Panin2}: For $M\in\Sm/k$, 
\[
A(\P^n\times M)\cong A(M)[t]/(t^{n+1})
\]
(with $t$ in even degree) the isomorphism sending $t$ to $c_1(\sO(1))$. Here $L\mapsto c_1(L)$ is the Chern structure associated to the given orientation.

\begin{rem}[\hbox{\cite[Corollary 3.3.8]{Panin2}}] \label{rem:ProjBundleFormSupp} Using the exact sequences of the pairs $(M,X)$ and  $(\P^n\times M,\P^n\times X)$, the projective bundle formula extends to give an isomorphism of $A_{X'}(M)$-modules (for any closed subset $X'$ of $M$ containing $X$)
\[
A_{\P^n\times X}(\P^n\times M)\cong A_X(M)\otimes_{A(M)}A(M)[t]/(t^{n+1})
\]
with $a\otimes t^i$ mapping to $p_2^*(a)\cup c_1(\sO(1))^i$. 
\end{rem}

We set $pt:=\Spec k$. Defining 
\[
A(\P^\infty\times M):=\lim_{\substack{\leftarrow\\N}}A(\P^N\times M)
\]
(using the system of inclusions $\P^N\to \P^{N+1}$ as the hyperplane $X_{N+1}=0$), we have
\[
A(\P^\infty\times M)\cong A(M)[[t]].
\]
Similarly,
\[
A(\P^\infty\times \P^\infty\times M)\cong A(M)[[u,v]],
\]
this latter isomorphism sending $u$ to $c_1(p_1^*\sO(1))$, $v$ to $c_1(p_2^*\sO(1))$. Thus, there is a well-defined element $F_A(u,v)\in A(pt)[[u,v]]$ with
\[
F_A(c_1(p_1^*\sO(1)), c_1(p_2^*(\sO(1)))=c_1(p_1^*\sO(1)\otimes p_2^*(\sO(1)).
\]
By Jouanolou's trick and functoriality, this gives
\[
F_A(c_1(L),c_1(L'))=c_1(L\otimes L')
\]
for each pair of line bundles $L, L'$ on some $M\in\Sm/k$. The fact that the set of isomorphism class of line bundles on $M\in\Sm/k$ is a group under tensor product  directly implies that $F_A(u,v)$ defines a (commutative, rank one) formal group law over $A(pt)$:
\begin{enumerate}
\item $F_A(u,0)=F_A(0,u)=u$.
\item $F_A(u,v)=F_A(v,u)$
\item $F_A(F_A(u,v), w)=F_A(u,F_A(v,w))$
\end{enumerate}
with inverse given by the power series $I_A(t)\in A(pt)[[t]]$ corresponding to $c_1(\sO(-1))$ under the isomorphism $A(\P^\infty)\cong A(pt)[[t]]$. 

\begin{rem} Since $c_1(L)$ has even degree, all the coefficients of $F_A(u,v)$ and $I_A(t)$ have even degree, so we actually have a formal group law over the commutative ring $A^{ev}(pt)$.
\end{rem}

For a commutative ring $R$, let $\Omega_{R[[t]]/R}^{ctn}:=\Omega_{R[t]/R}\otimes_RR[[t]]$. Given a commutative formal group law $F(u,v)\in R[[u,v]]$ over $R$, there is a unique normalized invariant differential form $\omega_F\in \Omega_{R[[t]]/R}^{ctn}$. We write $\omega_A$ for $\omega_{F_A}\in \Omega^{ctn}_{A^{ev}(pt)[[t]]/A^{ev}(pt)}$. Using the canonical generator $dt$ for $\Omega^{ctn}_{A^{ev}(pt)[[t]]/A^{ev}(pt)}$, we have
\[
\omega_A=(1+\sum_{n\ge1} a_nt^n)dt =  dt+a_1tdt +\ldots
\]
with $a_n\in A^{ev}(pt)$ (here ``normalized" means the first term is $dt$, i.e., $a_0=1$).

\begin{definition}We denote the projection $\P^n\times M\to M$ by $p^n$.  For a smooth pair $(M,X)$, define the map
\[
p_*^n:A_{\P^n\times X}(\P^n\times M)\to A_X(M)
\]
by 
\[
p_*^n(a\otimes t^i):= a_{n-i}\cdot a.
\]
Here we use the isomorphism $A_{\P^n\times X}(\P^n\times M)\cong A_X(M)\otimes_{A(M)}A(M)[t]/(t^{n+1})$ given by the projective bundle formula, and the canonical $A^{ev}(pt)$-module structure on $A_X(M)$.
\end{definition}

\begin{rem}\label{rem:PaninAgree} If we forget supports, the map $p^n_*$ agrees with the map $p^n_{quil}$ defined in \cite[\S4.3]{Panin2}.
\end{rem}

\begin{rem}\label{rem:ProjectionModule} Since the coefficients $a_n$ of $\omega_A$ are in $A^{ev}$, $p^n_*$ is an $A_Z(M)$-module map for all closed subsets $Z\subset M$ (in the sense of definition~\ref{def:IntSupp}(2)); in particular, $p^n_*$ is  an $A(M)$-module map.
\end{rem}

\begin{rem} Using the evident modification of the push-forward map $p^n_*$, we have maps
\[
p^n_*:A_{X\times \P^n\times Y}(M\times\P^n\times N)\to A_{X\times Y}(M\times N)
\]
for smooth pairs $(M,X)$ and $(N,Y)$. Since the basis elements in the projective bundle formula are of even degree, we need not worry about the order of the factors.
\end{rem}

\begin{prop}\label{prop:Projection} Let $A$ be a $\Z/2$-graded oriented ring cohomology theory.
\begin{enumerate}
\item For $(M,X)$ a smooth pair, the following diagram commutes
\[
\xymatrix{
A_{\P^n\times\P^m\times X}(\P^n\times \P^m\times M)\ar[r]^-{p^m_*}\ar[d]_{p^n_*}&A_{\P^n\times X}(\P^n\times M)\ar[d]^{p^n_*}\\
A_{\P^m\times X}(\P^m\times M)\ar[r]_-{p^m_*}&A_X(M)}
\]
\item Let $f:(N,Y)\to (M,X)$ be a morphism in $\SP$. Then the diagram 
\[
\xymatrix{
A_{\P^n\times X}(\P^n\times M)\ar[d]_{p^n_*}\ar[r]^{f^*}&
A_{\P^n\times Y}(\P^n\times N)\ar[d]^{p^n_*}\\
A_X(M)\ar[r]_{f^*}&A_Y(N)}
\]
commutes.
\item Let $i:\P^n\to \P^m$ be a linear closed immersion, $(M,X)$ a smooth pair. Then the diagram
\[
\xymatrix{
A_{\P^n\times X}(\P^n\times M)\ar[r]^{(i\times\id)_*}\ar[d]_{p^n_*}&
A_{\P^m\times X}(\P^m\times M)\ar[d]^{p^m_*}\\
A_X(M)\ar@{=}[r]&A_X(M)}
\]
commutes.
\item Let $i:N\to M$ be a closed immersion, $X\subset M$, $Y\subset N$ closed subsets with $i(Y)\subset X$. Then the diagram
\[
\xymatrix{
A_{\P^n\times Y}(\P^n\times N)\ar[r]^{(\id\times i)_*}\ar[d]_{p^n_*}&A_{\P^n\times X}(\P^n\times M)
\ar[d]^{p^n_*}\\
A_Y(N)\ar[r]_{i_*}&A_X(M)}
\]
commutes.
\item Let $s:M\to \P^n\times M$ be a section to the projection, and let $X\subset M$ be a closed subset. Then $p^n_*\circ s_*=\id_{A_X(M)}$.
\item Let $X\subset X'$ be closed subsets of $M\in\Sm/k$. Then the diagram
\[
\xymatrixcolsep{70pt}
\xymatrix{
A_{\P^n\times (X'\setminus X)}(\P^n\times(M\setminus X))
\ar[d]_{p^n_*}\ar[r]^-{\del_{\P^n\times M,\P^n\times X',\P^n\times X}}
&A_{\P^n\times X}(\P^n\times M)\ar[d]^{p^n_*}\\
A_{X'\setminus X}(M\setminus X)\ar[r]_{\del_{M,X',X}}&A_X(M)
}
\]
commutes.
\end{enumerate}
\end{prop}

\begin{proof} The proofs of (1)-(3) are exactly as the proofs of the corresponding properties in \cite[\S4.5]{Panin2}, adding the supports throughout.  We give a proof of (4) that is different from the approach used in \cite{Panin2}. 

By the projective bundle formula, it suffices to check the commutativity on elements of 
$A_{\P^n\times Y}(\P^n\times N)$ of the form $t^m\times a=p^{n*}(t^m)\cup p_Y^*(a)$, $t=c_1(\sO_{\P^n}(1))$. Since the Gysin map $(\id\times i)_*$ is a $A(\P^n\times M)$-module map, we have
\begin{align*}
(\id\times i)_*(t^m\times a)&=p^{n*}(t^m)\cup(\id\times i)_*(p_N^*(a))\\
&=p^{n*}(t^m)\cup p_M^*(i_*(a)))\\
&=t^m\times i_*(a),
\end{align*}
the second identity following from proposition~\ref{prop:Gysin}(5). Thus
\begin{align*}
p^n_*((\id\times i)_*(t^m\times a))&=p^n_*(t^m\times i_*(a))\\
&=a_{n-m}\cdot i_*(a).
\end{align*}
On the other hand
\begin{align*}
i_*(p^n_*(t^m\times a))&=i_*(a_{n-m}\cdot a)\\
&=a_{n-m}\cdot i_*(a),
\end{align*}
the second identity following because $i_*$ is an $A(M)$-module map, hence an $A(pt)$-module map. This proves (4).

For (5), the case without supports (proven in \cite[\S 4.6]{Panin2}) gives in particular the identity
\[
p^n_*(s_*(1))=1\in A(M)
\]
where $1\in A(M)$ is the identity. Now take an arbitrary element $a\in A_X(M)$ and write $p$ for $p^n$. Using the fact that both $s_*$ and $p_*$ satisfy the projection formula (for $i_*$, this is just property (2) of proposition~\ref{prop:Gysin} and for $p^n_*$, this is remark~\ref{rem:ProjectionModule}), we have
\begin{align*}
p_*(s_*(a))&=p_*(s_*(s^*p^*(a)\cup 1))\\
&=p_*(p^*(a)\cup s_*(1))\\
&=a\cup p_*(s_*(1))\\
&=a\cup 1=a.
\end{align*}

Finally, (6) follows from the partial Leibniz rule for $\del$, which implies that $\del_{X,Z',Z}$  and $\del_{\P^n\times X\P^n\times Z',\P^n\times Z}$ are $A(pt)$-module maps, together with the naturality of $\del$ with respect to pull-back. Thus,
\begin{align*}
p^n_*(\del_{\P^n\times M,\P^n\times X',\P^n\times X}(t^m\times a))&=
p^n_*(t^m\cup\del_{\P^n\times M,\P^n\times X',\P^n\times X}(p^{n*}(a)))\\
&=p^n_*(t^m\cup p^{n*}(\del_{M,  X', X}(a)))\\
&=a_{n-m}\cdot \del_{M,  X', X}(a)\\
&= \del_{M,  X', X}(a_{n-m}\cdot a)\\
&=\del_{M,  X', X}(p^n_*(t^m\times a)).
\end{align*}
\end{proof}
\ \\
\\
{\bf Step 3}: {\em The general case}.
Let $f:(N,Y)\to (M,X)$ be a   morphism in $\SP'$. Factor $f:M\to N$ as $f=p\circ i$, with $i:N\to \P^n\times M$ a closed immersion, and $p=p^n:\P^n\times M\to M$ the projection. Define $f_*:A_Y(N)\to A_X(M)$ as the composition
\[
A_Y(N)\xrightarrow{i_*}A_{\P^n\times X}(\P^n\times M)\xrightarrow{p_*}A_X(M).
\]

\begin{thm}\label{thm:Pushforward} Let $A$ be a $\Z/2$-graded oriented ring cohomology theory on $\SP$. 
\begin{enumerate}
\item For a   morphism $f:(N,Y)\to (M,X)$ in $\SP'$, the morphism $f_*:A_Y(N)\to A_X(M)$ does not depend on the choice of factorization $f=p\circ i$.
\item For a   morphism $f=i:(N,Y)\to (M,X)$ in $\SP'$ with $i:N\to M$ a closed immersion, $f_*$ agrees with the Gysin morphism defined in  Step 1. For $f=p^n:(\P^n\times M,\P^n\times X)\to (M,X)$ the projection, $f_*$ agrees with the map $p^n_*$ defined in Step 2. 
\item For a projective morphism $f:N\to M$, the map $f_*:A(N)\to A(M)$ agrees with the map $f_*$ defined in \cite[\S4.7]{Panin2}.
\item The assignment $[f:(N,Y)\to (M,X)]\mapsto f_*:A_Y(N)\to A_X(M)$ defines an integration with supports on $A$ (definition~\ref{def:IntSupp}), subjected to the given orientation on $A$.
\end{enumerate}
\end{thm}

\begin{proof} The proof of (1) is exactly as in the proof of the analogous result \cite[Theorem 4.7.1]{Panin2}, adding the supports where needed. The statement (2) follows directly from (1), as we may take $n=0$ if $f$ is a closed immersion, and $i$ the identity (and $i_*=\id$ as well) if $f=p^n$.  (3) follows from proposition~\ref{prop:Gysin}(1) and remark~\ref{rem:PaninAgree}.

For (4), the proofs of (1), (3) and (4) in definition~\ref{def:IntSupp} are exactly as in the proof of 
 \cite[Theorem 4.7.1]{Panin2}, adding the supports. Definition~\ref{def:IntSupp}(2) follows from proposition~\ref{prop:Gysin}(2) and remark~\ref{rem:ProjectionModule}.

 Definition~\ref{def:IntSupp}(5) follows from 
proposition~\ref{prop:Gysin}(6), while (6) follows from proposition~\ref{prop:Gysin}(8) and proposition~\ref{prop:Projection}(6). Thus, 
\[
[f:(N,Y)\to (M,X)]\mapsto f_*:A_Y(N)\to A_X(M)]\mapsto f_*:A_Y(N)\to A_X(M)
\]
 defines an integration with supports on $A$.

To complete the proof, we need only check that the integration with supports is subjected to the given orientation, i.e., that for a line bundle $p:L\to M$ with zero-section $s$, the composition
\[
A_X(M)\xrightarrow{s_*}A_{p^{-1}(X)}(L)\xrightarrow{s^*}A_X(M)
\]
is cup product with $c_1(L)$. By (3) and \cite[Theorem 4.1.4]{Panin2}, this is the case for $X=M$; in particular
\[
s^*(s_*(1))=c_1(L)\in A(M). 
\]
In general, take $a\in A_X(M)$. Then 
\begin{align*}
s^*(s_*(a))&=s^*(s_*(s^*p^*(a)\cup 1))\\
&=s^*(p^*(a)\cup s_*(1))\\
&=a\cup s^*(s_*(1))\\
&=a\cup c_1(L),
\end{align*} 
as desired.
\end{proof}

\begin{proof}[Proof of theorem~\ref{thm:Int}] The existence of an integration with supports subjected to a given orientation $\omega$ on $A$ follows from theorem~\ref{thm:Pushforward}. For the uniqueness, suppose we have two integrations 
\[
(f:(N,Y)\to (M,X))\mapsto f^1_*, f^2_*:A_Y(N)\to A_X(M),
\]
both subjected to the same orientation $\omega$. Let $\omega_A=(1+\sum_{n\ge1}a_nt^n)dt$ be the normalized invariant one-form for the formal group law $F_A$.

By the uniqueness part of \cite[Theorem 4.1.4]{Panin2}, $f^1_*=f^2_*:A(N)\to A(M)$. In particular, taking $q:\P^n\to \Spec k$ to be the structure map, and  letting $t=c_1(\sO_{\P^n}(1))\in A(\P^n)$, we have
\[
q^1_*(t^m)=q^2_*(t^m)=a_{n-m}.
\]

Now let $(M,X)$ be a smooth pair,   take $a\in A_X(M)$, and let $p:\P^n\times M\to M$ be the projection. Then for $i=1,2$, 
\begin{align*}
p^i_*(t^m\times a)&= q_*^i(t^m)\cdot a\\
&=a_{n-m}\cdot a,
\end{align*}
so $p^1_*=p^2_*$. 

Next, consider a closed immersion $i:N\to M$ in $\Sm/k$, and let $Y\subset N$, $X\subset M$ be closed subsets with $i(Y)\subset X$. Suppose $i=\id_M$. Then, by definition~\ref{def:IntSupp}(5), $i_*^1=\id_M^*=i^2_*$; this reduces us to the case $X=i(Y)$.

Since $M$ is quasi-projective, we can find a sequence of smooth closed subschemes
\[
N=N_0\subset N_1\subset\ldots\subset N_r=M
\]
such that $N_{i-1}$ is a smooth codimension one subscheme of $N_i$ for $i=1,\ldots, r$. This reduces us to the case of a codimension one closed immersion. 

Consider the deformation to the normal bundle
\[
\xymatrix{
\nu\ar[r]^{i_0}&M_t&M\ar[l]_{i_1}\\
N\ar[r]_{i_0}\ar[u]^s&N\times\A^1\ar[u]_{\tilde{i}}&N\ar[l]^{i_1}\ar[u]_i\\
Y\ar@{^(->}[u]\ar[r]_{i_0}&Y\times\A^1\ar@{^(->}[u]&Y.\ar[l]^{i_1}\ar@{^(->}[u]}
\]
This gives us the commutative diagram (for $j=1,2$)
\[
\xymatrix{
A_Y(N)\ar[d]_{s^j_*}&A_{Y\times\A^1}(N\times\A^1)\ar[l]_-{i_0^*}\ar[r]^-{i_1^*}\ar[d]_{\tilde{i}^j_*}
&A_Y(N)\ar[d]^{i^j_*}\\
A_Y(\nu)&A_{Y\times\A^1}(M_t)\ar[l]^-{i_0^*}\ar[r]_-{i_1^*}&A_Y(M).
}
\]
 It follows easily from the homotopy property for $A$ that the maps $i_0^*, i_1^*$ in the upper row are isomorphisms; the maps $i_0^*,i_1^*$ in the lower row are isomorphisms by lemma~\ref{lem:Htpy}. Thus, it suffices to show that $s^1_*=s^2_*$.

Using excision, we can replace $\nu$ with the $\P^1$-bundle $\P(\nu\oplus O_N)$. Let $p:\P(\nu\oplus O_N)\to N$ be the projection. Clearly $p_*^j:A_Y(\P(\nu\oplus O_N))\to A_Y(N)$ is inverse to $s^j_*$, $j=1,2$, so it suffices to see that
\[
p_*^1=p_*^2:A_Y(\P(\nu\oplus O_N))\to A_Y(N).
\]
The map $p_*^j$ factors through  the ``enlarge the support map"
\[
\id^*:A_Y(\P(\nu\oplus O_N))\to A_{p^{-1}(Y)}(\P(\nu\oplus O_N))
\]
hence $\id^*$ is injective.  Thus, it suffices to see that the maps
\[
p_*^j:A_{p^{-1}(Y)}(\P(\nu \oplus O_N))\to A_Y(N)
\]
are equal.   

We have the extended projective bundle formula \cite[Corollary 3.3.8]{Panin2}:  Let
\[
\alpha:A_Y(N)\oplus A_Y(N)\to A_{p^{-1}(Y)}(\P(\nu\oplus O_N))
\]
be the map sending $(a, b)$ to $p^*(a)+p^*(b)\cup c_1(\sO(1))$. Then $\alpha$ is an isomorphism. \\
\\
The projection formula implies
\[
p^j_*(\alpha(a,b))=a\cup p^j_*(1_{\P(\nu\oplus O_N)})+b\cup p^j_*(c_1(\sO(1))).
\]
By the uniqueness part of \cite[Theorem 4.1.4]{Panin2}, 
\[
p_*^1=p_*^2:A(\P(\nu\oplus O_N))\to A(N),
\]
hence $p^1_*(1_{\P(\nu\oplus O_N)})=p^2_*(1_{\P(\nu\oplus O_N)})$ and $p^1_*(c_1(\sO(1)))=p^2_*(c_1(\sO(1)))$. Thus, $p^1_*=p_*^2:A_{p^{-1}(Y)}(\P(\nu\oplus O_N))\to
A_Y(N)$. 

Since each projective morphism $f$ factors as $p\circ i$, the two cases of a projection and a closed immersion imply $f_*^1=f_*^2$ for all $f$, completing the proof.
\end{proof}

\section{Algebraic oriented cohomology}\label{sec:Mocanasu}  Mocanasu \cite{Mocanasu}  has considered a version of oriented cohomology, with somewhat different axioms from what we have discussed so far, and has shown that such a theory gives rise to a Borel-Moore homology theory on quasi-projective schemes (over a fixed base-field $k$). Shortly speaking, the Borel-Moore homology theory $H$ corresponding to an oriented ring cohomology theory $A$ is given by 
\[
H(X):=A_X(M)
\]
for any smooth pair $(M,X)$. The main point is to show that this is independent of the choices, both in the smooth ``envelope" $M$ for a given $X$, as well as for morphisms $F:M\to M'$ extending a given projective morphism $f:X\to X'$. In this section, we give a modified version of Mocanasu's notion of an algebraic oriented theory, and  show that the integration with supports defined on an oriented ring cohomology theory satisfies the axioms (assuming that the base-field admits resolution of singularities). We fix a base-field $k$ and an oriented $\Z/2$-graded ring cohomology theory $A$ on $\SP$. We let $\Sch_k$ denote the category of quasi-projective $k$-schemes and $\Sch_k'$ the subcategory with the same objects, but with only the projective morphisms.   We will assume throughout this section that $k$ admits resolution of singularities.

Let $(M,X)$, $(N,Y)$ be smooth pairs, and let $F:M\to N$ be a morphism such that $F(X)\subset Y$ and the restriction $f:X\to Y$ of $F$ is projective. Let   $\sC_F$ be the category of all dense open immersions $j:M\to\bar{M}$, with $\bar{M}\in\Sm/k$, and extensions $\bar{F}:\bar{M}\to N$ such that $\bar{F}$ is projective; a morphism $\mu:(j:M\to\bar{M},\bar{F})\to (j':M\to\bar{M}',\bar{F}')$ is a morphism $\mu:\bar{M}\to\bar{M}'$ with $j'=\mu\circ j$ and $\bar{F}'\circ \mu=\bar{F}$. Note that $\mu$ is necessarily projective and birational. Also, since $f$ is projective, $j(X)$ is closed in $\bar{M}$.

\begin{lem}\label{lem:Filt} The category $\sC_F$ is left filtering, and there is 
 is at most one morphism between any two objects.
\end{lem}

\begin{proof} This follows easily from resolution of singularities. $\sC_F$ is non-empty: since $M$ is quasi-projective, the map $F$ factors through a locally closed immersion $M\to \P^n\times N$.  We can close up $M$ in $\P^n\times N$ and resolve singularities to construct $\bar{M}$, $j$ and $F$.

Given two objects in $\sC_F$, $\alpha_1:=(j_1:M\to \bar{M}_1,  \bar{F}_1:\bar{M}_1\to N)$, 
$\alpha_2:=(j_2:M\to \bar{M}_2,\bar{F}_2:\bar{M}_2\to N)$, resolve the singularities of the closure of $(j_1,j_2)(M)$ in $M_1\times_k M_2$ to construct $j_3:M\to \bar{M}_3$, $\bar{F}_3:\bar{M}_3\to N$ 
dominating $\alpha_1$ and $\alpha_2$.  Since $M$ is assumed dense in $\bar{M}$, there is at most one morphism between any two objects of $\sC_F$, completing the proof.
\end{proof}

\begin{definition}\label{def:Pushforward}
Let $(M,X)$, $(N,Y)$ be smooth pairs, and let $F:M\to N$ be a morphism such that $F(X)\subset Y$ and the restriction $f:X\to Y$ of $F$ is projective. Define the push-forward morphism
\[
F_*:A_X(M)\to A_Y(N)
\]
by taking $(j:M\to\bar{M}, \bar{F}:\bar{M}\to N)$ in $\sC_F$ and setting
\[
F_*:=\bar{F}_*\circ (j^*)^{-1},
\]
where $j^*:A_{j(X)}(\bar{M})\to A_X(M)$ is the excision isomorphism. 
\end{definition}
We note that $F_*$ is well-defined by lemma~\ref{lem:Filt}.

\begin{lem}\label{lem:Funct} Given composable morphisms $F:M\to N$, $G:N\to P$, and smooth pairs $(M,X)$, $(N,Y)$, $(U,Z)$ suppose that $F(X)\subset Y$, $G(Y)\subset Z$, and that the restrictions of $F$ and $G$, $f:X\to Y$, $g:Y\to Z$, are projective. Then
\[
G_*\circ F_*=(G\circ F)_*:A_X(M)\to A_Z(P).
\]
\end{lem}

\begin{proof} Take $j_1:N\to\bar{N}, \bar{G}:\bar{N}\to P$ in $\sC_G$,  $j_2:M\to\bar{M}, \bar{F}:\bar{M}\to N$ in $\sC_F$ and $j_3:\bar{M}\to\bar{M}', \bar{F}':\bar{M}'\to \bar{N}$ in $\sC_{j\circ \bar{F}}$. Then $j_3\circ j_2:M\to\bar{M}',\bar{G}\circ \bar{F}':\bar{M}'\to P$ is in $\sC_{G\circ F}$, so
\begin{align*}
&(G\circ F)_*=\bar{G}_*\circ\bar{F}'_*\circ(j^*_2\circ j^*_3)^{-1}\\
&F_*=\bar{F}_*\circ(j_2^*)^{-1}\\
&G_*=\bar{G}_*\circ (j_1^*)^{-1}.
\end{align*}
Since the diagram
\[
\xymatrix{
\bar{M}\ar[r]^{j_3}\ar[d]_{\bar{F}}&\bar{M}'\ar[d]^{\bar{F}'}\\
N\ar[r]_{j_1}&\bar{N}
}
\]
is transverse cartesian, we have (by lemma~\ref{lem:PushPull})
\[
j_1^*\circ\bar{F}'_*=\bar{F}_*\circ j_3^*.
\]
Thus
\begin{align*}
G_*\circ F_*&=\bar{G}_*\circ (j_1^*)^{-1}\circ \bar{F}_*\circ(j_2^*)^{-1}\\
&=\bar{G}_*\circ \bar{F}'_*\circ( j_3^*)^{-1}\circ (j_2^*)^{-1}\\
&=(G\circ F)_*
\end{align*}
\end{proof}

\begin{prop}\label{prop:Extension} Let $(M,X)$, $(M',X')$ be smooth pairs, $F,G:M\to M'$ two   morphisms such that $F$ and $G$ restrict to the same projective morphism $f:X\to X'$. Then
\[
F_*=G_*:A_X(M)\to A_{X'}(M').
\]
\end{prop}

\begin{proof}   We first reduce to the case of affine $M'$. Indeed, Jouanolou tells us that there is an affine space bundle $q:E\to M'$ with $E$ affine. Since $E\to M'$ is smooth, replacing $X'$ with $q^{-1}(X')$,  $M$ with $M\times_{M'}E$ and $X$ with $X\times_{M'}E$, and using the extended homotopy property (remark~\ref{rems:Cohomology}(2)) achieves the reduction.

Next, we reduce to the case in which $M'=\A^n$ for some $n$. Since $M'$ is affine, there  a closed immersion $i:M'\to \A^n$. By remark~\ref{rem:GysinIso}, the push-forward $i_*:A_{X'}(M')\to A_{i(X')}(\A^n)$ is an isomorphism, so we may replace $(M',X')$ with $(\A^n,i(X'))$, and change notation.

Consider the product map
\[
(F,G):M\to \A^n\times_k \A^n.
\]
Since $F$ and $G$ are both equal to $f$ when restricted to $X$, we have the commutative diagram
\[
\xymatrix{
X\ar@{^(->}[dd]\ar[r]^f&X'\ar@{^(->}[d]\\
&\A^n\ar[d]^\delta\\
M\ar[r]_-{(F,G)}&\A^n\times \A^n
}
\]
where $\delta$ is the diagonal. 

Consider the map
\begin{align*}
\phi:\A^1\times \A^n\times \A^n&\to \A^n\times \A^n\\
\phi(t,x,y)&:=(t,x, ty+(1-t)x).
\end{align*}
For $a\in k$,  let $\phi_a:\A^n\times \A^n\to \A^n\times \A^n$ be the fiber of $\phi$ over $a$.  Note that
\[
\phi\circ(\id\times \delta)=\id\times\delta:\A^1\times\A^n\to \A^1\times\A^n\times\A^n.
\]
Thus, we may form the commutative diagram of schemes over $\A^1$
\[
\xymatrixcolsep{50pt}
\xymatrix{
\A^1\times X\ar@{^(->}[dd]\ar[r]^{\id\times f}&\A^1\times X'\ar@{^(->}[d]\\
&\A^1\times \A^n\ar[d]^{\id\times \delta}\\
\A^1\times M\ar[r]_-{\phi\circ[\id\times(F,G)]}&\A^1\times \A^n\times \A^n
}
\]

Let $j:\A^n\to\P^n$ be the standard open immersion.  Since  $M$ is quasi-projective, there is an open immersion $g:U\hookrightarrow \P^N$ for some $N$, and a closed immersion $i:M\to U$. Thus, we may
 factor $\phi\circ[\id\times(F,G)]$ as a composition of maps over $\A^1$
 \[
 \A^1\times M\xrightarrow{\iota}\A^1\times U\times \A^n\times \A^n\xrightarrow{\gamma}
 \A^1\times\P^N\times\A^n\times\A^n\xrightarrow{q} \A^1\times\A^n\times\A^n,
 \]
 with $\iota$ a closed immersion, $\gamma=\id\times g\times\id$ and $q$ the projection. Let $\sM^*$ be the closure of $\gamma\circ\iota(\A^1\times M)$ in $\A^1\times\P^N\times\P^n\times\A^n$ and let $\mu:\sM\to \sM^*$ be a resolution of singularities of $\sM^*$. We note that $\sM^*$, and hence $\sM$,   is naturally a scheme over $\A^1$, and similarly, the map  $\phi\circ[\id\times(F,G)]$ extends to a map
 \[
 \pi:\sM\to \A^1\times\P^n\times\A^n
 \]
 over $\A^1$.  Since $\gamma\circ\iota(\A^1\times M)$ is a smooth dense open subscheme of $\sM^*$, we may take $\sM$ so that $\mu:\sM\to \sM^*$ is an isomorphism over $\gamma\circ\iota(\A^1\times M)$; let $h:\A^1\times M\to \sM$ be the resulting open dense immersion.

 This gives us the commutative diagram of schemes over $\A^1$
 \[
\xymatrixcolsep{50pt}
\xymatrix{
\A^1\times X\ar@{^(->}[dd]\ar[r]^{\id\times f}&\A^1\times X'\ar@{^(->}[d]\\
&\A^1\times \A^n\ar[d]^{\id\times \delta}\\
\A^1\times M\ar@{^(->}[d]_h\ar[r]_-{\phi\circ[\id\times(F,G)]}&\A^1\times \A^n\times \A^n\ar[d]^\gamma\\
\sM\ar[r]_-\pi\ar[d]_p&\A^1\times \P^n\times \A^n\ar[d]^{p_1}\\
\A^1\ar@{=}[r]&\A^1}
\]
 
  For $a\in k$, let $\pi_a:\sM_a\to \P^n\times\A^n$ be the fiber of $\pi$ over $a$.   We note that $\sM_a$ contains $M$ as an open subscheme, and that $\pi_a$ extends $\phi_a\circ[\id\times(F,G)]$. We let $\bar{M}_a\subset \sM_a$ be the closure of $M$ in $\sM_a$. Blowing up $\sM$ further  and changing 
 notation if necessary, we may assume that $\bar{M}_0$ and $\bar{M}_1$ are smooth.  Let   $\iota_0:\bar{M}_0\to\sM$, $\iota_1:\bar{M}_1\to\sM$ denote the inclusions.

Noting that $\delta(\A^n)$ is closed in $\P^n\times\A^n$, we see that $\A^1\times \delta(X')$ is closed in $\A^1\times\P^n\times\A^n$. Since $\id\times f:\A^1\times X\to \A^1\times X'$ is projective, this implies that 
$h(\A^1\times X)$ is closed in $\sM$, and $h(0\times X)$, $(1\times M)$ are thus closed and contained in $\bar{M}_0\subset p^{-1}(0)$, $\bar{M}_1\subset p^{-1}(1)$, respectively.   We have the commutative diagram
\[
\xymatrix{
A_{h(0\times X)}(\bar{M}_0)\ar[d]_{h_0^*}&A_{h(\A^1\times X)}(\sM)\ar[l]_-{\iota_0^*}\ar[r]^-{\iota_1^*}\ar[d]^{h^*}&A_{h_1(1\times X)}(\bar{M}_1)\ar[d]^{h_1^*}\\
A_X(M)&A_{\A^1\times X}(\A^1\times M)\ar[l]^-{i_0^*}\ar[r]_-{i_1^*}&A_X(M)
}
\]
where $i_0, i_1:M\to\A^1\times M$ are the 0-, 1-sections, and $h_0$, $h_1$ are the restrictions of $h$. 

By the homotopy property for $A$, the maps $i_0^*, i_1^*$ are isomorphisms and $i_0^*=i_1^*$. The maps $h$, $h_0^*$ and $h_1^*$ are isomorphisms by excision.  

Since $h(\A^1\times M)\subset \sM$ is an open neighborhood of $h(\A^1\times X)$ in $\sM$ that is smooth over $\A^1$, we may apply lemma~\ref{lem:PushPull} to give the commutative diagram
\[
\xymatrix{
A_X(M)&A_{\A^1\times X}(\A^1\times M)\ar[l]_-{i_0^*}\ar[r]^-{i_1^*}&A_X(M)\\
A_{h(0\times X)}(\bar{M}_0)\ar[u]^{h_0^*}\ar[d]_{\pi_{0*}}&A_{h(\A^1\times X)}(\sM)\ar[u]_{h^*}
\ar[l]_-{\iota_0^*}\ar[r]^-{\iota_1^*}\ar[d]^{\pi_*}&A_{h_1(1\times X)}(\bar{M}_1)\ar[d]^{\pi_{1*}}\ar[u]_{h_1^*}\\
A_{\delta(X')}(\P^n\times\A^n)&A_{\A^1\times \delta(X')}(\A^1\times\P^n\times\A^n)
\ar[l]^-{i_0^*}\ar[r]_-{i_1^*}&A_{\delta(X')}(\P^n\times\A^n)
}
\]
Since $i_1^*\circ(i_0^*)^{-1}=\id$ (for both the top row and the bottom row), this gives
\[
\pi_{0*}\circ (h^*_0)^{-1}=\pi_{1*}\circ(h_1^*)^{-1}:A_X(M)\to A_{\delta(X')}(\P^n\times\A^n).
\]
Composing with the push-forward for the projection $p_2:\P^n\times\A^n\to \A^n$, we have
\begin{equation}\label{eqn:1}
p_{2*}\circ \pi_{0*}\circ (h^*_0)^{-1}=p_{2*}\circ\pi_{1*}\circ(h_1^*)^{-1}:A_X(M)\to A_{X'}(\A^n).
\end{equation}

Since 
\[
p_2\circ \phi\circ[\id\times(F,G)]\circ i_0=F,\quad p_2\circ \phi\circ[\id\times(F,G)]\circ i_1=G,
\]
we have commutative diagrams
\[
\xymatrix{
M\ar[d]_{h_0}\ar[r]^F&\A^n \\
\bar{M}_0\ar[ur]_{p_2\circ\pi_0}}\quad
\xymatrix{
M\ar[d]_{h_1}\ar[r]^G&\A^n\\
\bar{M}_1\ar[ur]_{p_2\circ\pi_1}}
\]
Thus, $(h_0:M\to\bar{M}_0, p_2\circ\pi_0)$ is in $\sC_F$ and 
$(h_1:M\to\bar{M}_0, p_2\circ\pi_1)$ is in $\sC_G$, hence
\begin{align*}
F_*&=(p_2\circ\pi_0)_*\circ(h_0^*)^{-1 }\\
G_*&=(p_2\circ\pi_1)_*\circ \circ(h_0^*)^{-1 }.
\end{align*}
Together with \eqref{eqn:1}, this gives $F_*=G_*$.
\end{proof}

\begin{lem}\label{lem:IdPush} Let $F:M\to N$ be a   morphism in $\Sm/k$, $(M,X)$, $(N,Y)$ smooth pairs. Suppose that $F(X)=Y$ and that the restriction of $F$ to $f:X\to Y$ is an isomorphism (using the reduced scheme structures). Then $F_*:A_X(M)\to A_Y(N)$ is an isomorphism.
\end{lem}

\begin{proof} Extending $F$ to $\bar{F}:\bar{M}\to N$ for some $(M\hookrightarrow\bar{M}, \bar{F})$ in $\sC_F$, and changing notation, we may assume that $F$ is projective. Factoring $F$ as $F=p\circ i$, with $p:\P^n\times N\to N$ the projection and $i:M\to \P^n\times M$ a closed immersion, it suffices to handle the two cases $F=i$ and $F=p$.

For $F=i$, this is remark~\ref{rem:GysinIso}. In the case of a projection, let $s:Y\to \P^n\times N$ be the section induced by the isomorphism $p:X\to Y$. Suppose we have  an extension of $s$ to a section $t:N\to \P^n\times N$. Letting $M:=t(N)$, with closed immersion $i:M\to \P^n\times N$, we have the commutative diagram
\[
\xymatrix{
A_X(M)\ar[r]^{i_*}\ar[rd]_{(p\circ i)_*}&A_X(\P^n\times N)\ar[d]^{p_*}\\
&A_Y(N).
}
\]
As $p\circ i:(M,X)\to (N,Y)$ is an isomorphism of smooth pairs, the map $(p\circ i)_*:A_X(M)\to
A_Y(N)$ is an isomorphism. From the case of a closed immersion, $i_*:A_X(M)\to A_X(\P^n\times N)$ is also an isomorphism, hence $p_*:A_X(\P^n\times N)\to A_Y(N)$ is an isomorphism as well.

We claim that $N$ admits a Zariski open cover
\[
N=\cup_{i=1}^sU_i
\]
such that the restriction of $s$ to $U_i\cap Y$ extends to a section $t_i:U_i\to \P^n\times U_i$. Using Mayer-Vietoris and the case in which a section extends, handled above, this will prove the result in general. To prove our claim, let $y$ be a point of $Y$. Shrinking $N$ to some affine neighborhood $U$ of $y$, we may assume that $s(Y)$ is contained in a product $\A^n\times N$, where $\A^n$ is some standard affine subset of $\P^n$. The map $s$ is then given by a morphism $\bar{s}:Y\to \A^n$, i.e., by $n$ regular functions $\bar{s}_1,\ldots, \bar{s}_n$ on $Y$. As $U$ is affine, each $\bar{s}_i$ lifts to a regular function $\bar{t}_i$ on $U$, giving the desired section $t:U\to \A^n\times U\subset \P^n\times U$ extending $s$.
\end{proof}

We can also extend the compatibility of push-forward with the boundary in the long exact sequence of a pair (definition~\ref{def:IntSupp}(5))

\begin{lem}\label{lem:BoundaryExtension} Let $F:M\to N$ be a   morphism in $\Sm/k$, let $X\subset X'\subset M$, $Y\subset Y'\subset N$ be closed subsets. Suppose that  $F(X')\subset Y'$, that the restriction of $F$ to $f:X'\to Y'$ is projective and that $f^{-1}(Y)\cap X'=X$. Then the diagram
\[
\xymatrixcolsep{40pt}
\xymatrix{
A_{X'\setminus X}(M\setminus X)\ar[r]^-{\del_{M,X',X}}\ar[d]_{F_*}&A_X(M)\ar[d]^{F_*}\\
A_{Y'\setminus Y}(N\setminus Y)\ar[r]_-{\del_{N,Y',Y}}&A_Y(N)}
\]
commutes. 
\end{lem}

\begin{proof} Just take $M\to \bar{M}, \bar{F}:\bar{M}\to N$ in $\sC_f$ and apply definition~\ref{def:IntSupp}(5).
\end{proof}

We give a modified version of Mocanasu's notion of an  algebraic oriented theory on $\SP$. In what follows, for $(M,X)$ a smooth pair, we consider $X$ as a scheme by given it the reduced structure.

\begin{definition}\label{def:AlgOrient} An  {\em algebraic oriented theory} on $\SP$ consists of the following data:
\begin{enumerate}
\item[(D1)]  A functor  $A:\SP^\op\to \Ab$. For a morphism $G:(M,X)\to (N,Y)$ in $\SP$,  we write $G^*:A_Y(N)\to A_X(M)$ for $A(G)$.\\
\item[(D2)] Let $(M,X)$, $(N,Y)$ be smooth pairs, let $F:M\to N$ be a morphism such that $F(X)\subset Y$ and such that the restriction of $F$ to $f:X\to Y$ is projective. Then there is a ``push-forward" $F_*:A_X(M)\to A_Y(N)$.
\end{enumerate}
These satisfy:
\begin{enumerate}
\item[(A1)] $A$ is additive: For $X, Y\in\Sm/k$, the canonical map $A(X\amalg Y)\to A(X)\times A(Y)$ is an isomorphism.\\
\item[(A2)]  i. Let $(M,X)$, $(N,Y)$ be smooth pairs, let $F, G:M\to N$ be morphisms in $\Sm/k$  such that $F(X)\subset Y$, $G(X)\subset Y$ and such that $F$ and $G$ restrict to the same projective morphism $f:X\to Y$. Then $F_*=G_*:A_X(M)\to A_Y(N)$.\\
\\
ii. Let $(M,X)$, $(N,Y)$ be smooth pairs, let $F:M\to N$ be a morphism such that $F(X)\subset Y$ and such that the restriction of $F$ to $f:X\to Y$ is an isomorphism. Then $F_*:A_X(M)\to A_Y(N)$ is an isomorphism.\\
\item[(A3)] Given smooth pairs $(M_1, X_1)$, $(M_2,X_2)$ and $(M_3, X_3)$ and a commutative diagram
\[
\xymatrix{
X_1\ar@{^(->}[d]\ar[r]^{f_1}&X_2\ar@{^(->}[d]\ar[r]^{f_2}&X_3\ar@{^(->}[d]\\
M_1\ar[r]_{F_1}&M_2\ar[r]_{F_2}&M_3
}
\]
such that $f_1$ and $f_2$ are projective, then $F_{2*}\circ F_{1*}=(F_2\circ F_1)_*$.\\
\item[(A4)] Suppose we have  smooth pairs  $(M, X)$, $(M',X')$, $(N, Y)$  and $(N', Y')$, and a commutative diagram
\[
\xymatrix{
&X'\ar[rr]^{f'}\ar[dl]_{g'}\ar@{^(->}[dd]|\hole&&Y'\ar[dl]^g\ar@{^(->}[dd]\\
X\ar[rr]^(.6){f}\ar@{^(->}[dd]&&Y\ar@{^(->}[dd]\\
&M'\ar'[r]_(.7){F'}[rr]\ar[dl]_{G'}&&N'\ar[dl]^G\\
M\ar[rr]_F&&N
}
\]
such that the top, bottom, left and right squares are cartesian, and that the bottom square is  transverse. Suppose further that $f$ and $f'$ are projective. Finally, suppose that $G$ and $G'$ are either smooth and equi-dimensional, or    closed immersions. Then
\[
F'_*\circ G^{\prime*}=G^*\circ F'_*:A_X(M)\to A_{Y'}(N').
\]
\item[(A5)] Let $(M,X)$, $(N,X)$, $(M',X)$ and $(N', X')$ be smooth pairs. Suppose we have a cartesian diagram
\[
\xymatrix{
M\ar[r]^{G'}\ar[d]_{F'}&M'\ar[d]^F\\
N\ar[r]_G&N'
}
\]
with $X=G^{-1}(X')$, $X=F^{-1}(X')$, $G'(X)=X$, $F'(X)=X$ and such that the restrictions $G':X\to X$, $F':X\to X$ are the identity. Suppose  that $F$ and $G$ are open immersions. We have the diagram
\[
\xymatrix{
A_X(M)\ar[r]^{G'_*}\ar[d]_{F'_*}&A_X(M')\\
A_X(N)&A_{X'}(N');\ar[l]^{G^*}\ar[u]_{F^*}
}
\]
note that $G'_*$ and $F'_*$ are isomorphisms by (A2)(ii). Then 
\[
(G'_*)^{-1}\circ F^*=(F'_*)^{-1}\circ G^*.
\]
\item[(A6)] Let 
\[
\xymatrix{
Z\ar@{^(->}[r]\ar[d]&V\ar[d]^p\\
Z'\ar@{^(->}[r]&X
}
\]
be a cartesian diagram, where the horizontal arrows are inclusions of  reduced closed subschemes, and $p:V\to X$ is an affine space bundle. Then $p^*:A_{Z'}(X)\to A_Z(V)$ is an isomorphism.\\
\item[(A7)] Let $(X,Z)$ be a smooth pair. Then $\id_{X*}:A_Z(X)\to A_Z(X)$ is the identity map.\\
\item[(A8)] Let $X\subset Y\subset M$ be closed subsets of $M\in \Sm/k$. Then 
\[
\id_{M*}=\id_M^*:A_X(M)\to A_Y(M).
\]
\end{enumerate}
\end{definition}

\begin{rems} Other than notational or organizational changes, our axioms for an oriented algebraic theory differ from Mocanasu's notion \cite[Definition 1.15]{Mocanasu} of an oriented algebraic theory at the following  points:\\
\\
1. Mocanasu's axiom (A4) differs from ours in that she does not assume that the bottom square is cartesian, and does not require the bottom square to be transverse if $G$ and $G'$ are closed immersions. However, in all uses of (A4) in \cite{Mocanasu}, the bottom square {\em is} transverse cartesian, so this does not lead to any difference in the applications.\\
\\
2. Mocanasu's axiom (A5) differs from ours in that she allows the morphisms $F$ and $G$ to be smooth and equi-dimensional, rather than requiring them to be open immersions. This causes a difference in the associated Borel-Moore homology theories, in that our Borel-Moore homology theories will only have functorial pull-back morphisms for open immersions, whereas the Borel-Moore homology theories of Mocanasu have functorial pull-back morphisms for smooth equi-dimensional morphisms that are ``embeddable". \\
\\
3. We have strengthened the homotopy axiom (A6) from that of \cite{Mocanasu}, by allowing $V$ to be an affine space bundle rather than a vector bundle.\\
\\
4. We have added the axiom (A8), which appears as an additional condition on an algebraic oriented theory in the statement of   \cite[Proposition 4.2]{Mocanasu}. 
\end{rems}

\begin{thm} \label{thm:AlgOrient} Suppose  that $k$ admits resolution of singularities.
 Let $A$ be an oriented $\Z/2$-graded ring cohomology theory  on $\SP$. Then the functor $A:\SP^\op\to \Ab$ (forget the $\Z/2$-grading) and the push-forward maps of definition~\ref{def:Pushforward} define an algebraic oriented theory on $\SP$.
\end{thm}

\begin{proof} We are using the integration with supports on $A$ given by theorem~\ref{thm:Int}. The axiom (A1) follows from Mayer-Vietoris. (A2)(i) is proposition~\ref{prop:Extension}, (A2)(ii) is 
lemma~\ref{lem:IdPush} and (A3) is lemma~\ref{lem:Funct}. (A4) follows from lemma~\ref{lem:PushPull} and (A6) follows from the homotopy property for $A$ together with Mayer-Vietoris, (A7) follows from (A8) and (A8) is definition~\ref{def:IntSupp}(4). The axiom  (A5) follows from the functoriality of pull-back, together with the identities
\[
(G'_*)^{-1}=G^{\prime*},\quad (F'_*)^{-1}=F^{\prime*}.
\]
\end{proof}

\section{Oriented duality theories}\label{sec:OrientDualityThy}  We describe an analog of  Bloch-Ogus twisted duality theory \cite{BlochOgus} for oriented cohomology.   As in the previous section, we will assume that
the base-field $k$ admits resolution of singularities, although this assumption is not needed for definition~\ref{def:OrientedDuality}.

\begin{definition}\label{def:OrientedDuality} An {\em oriented duality theory} $(H,A)$ on $\Sch_k$ consists of
\begin{enumerate}
\item[(D1)] A functor $H:\Sch'_k\to \gr_{\Z/2}\Ab$.
\item[(D2)] A $\Z/2$-graded oriented ring cohomology theory $A$ on $\SP$.
\item[(D3)] For each open immersion  $j:Y\to X$ in $\Sch_k$, a map $j^*:H(X)\to H(Y)$.
\item[(D4)] i. For each smooth pair $(M,X)$, and each  morphism $f:Y\to M$ in $\Sch_k$, a graded cap product map
\[
f^*(-)\cap:A_X(M)\otimes H(Y)\to H(Y\cap f^{-1}(X)).
\]
ii. For $X,Y\in\Sch_k$, a graded external product
\[
\times:H(X)\otimes H(Y)\to H(X\times Y).
\]
\item[(D5)] For each smooth pair $(M,X)$, an isomorphism
\[
\alpha_{M,X}:H(X)\to A_X(M).
\]
\item[(D6)] For $X\in \Sch_k$ and for   $Y\subset X$ a closed subset, a degree 1 map
\[
\del_{X,Y}:H(X\setminus Y)\to H(Y).
\]
\end{enumerate}
We let $[F:(M,X) \to (N,Y)] \text{ in } \SP'\mapsto F_*:A_X(M)\to A_Y(N)$ be the integration with supports on $A$ subjected to the given orientation. The data (D1)-(D6) satisfy
\begin{enumerate}
\item[(A1)] Let $(M,X)$, $(N,Y)$ be smooth pairs, and let $j:M\to N$ be an open immersion with $j^{-1}(Y)=X$. Let $j_Y:X\to Y$ be the restriction of $j$. Then the diagram
\[
\xymatrix{
H(Y)\ar[r]^{\alpha_{N,Y}}\ar[d]_{j_Y^*}&A_Y(N)\ar[d]^{j^*}\\
H(X)\ar[r]_{\alpha_{M,X}}&A_X(M)
}
\]
commutes.
\item[(A2)] Let $(M,X)$, $(N,Y)$ be smooth pairs, let $f:X\to Y$ be a projective morphism in $\Sch_k$, and suppose $f$ extends to a projective morphism $F:M\to N$. Then the diagram
\[
\xymatrix{
H(X)\ar[r]^{\alpha_{M,X}}\ar[d]_{f_*}&A_X(M)\ar[d]^{F_*}\\
H(Y)\ar[r]_{\alpha_{N,Y}}&A_Y(N)
}
\]
commutes.
\item[(A3)] Let $(M,X)$ and $(N,Y)$ be smooth pairs.\\
i. let  $F:N\to M$ be a morphism in $\Sm/k$, and let $f:Y\to M$ be the restriction of $F$.   Let
\[
F^*(-)\cup:A_X(M)\otimes A_Y(N)\to A_{Y\cap f^{-1}(X)}(N)
\]
be the map $a\otimes b\mapsto F^*(x)\cup y$, with $F^*:A_X(M)\to A_{f^{-1}(X)}(N)$ the pull-back. Then the diagram
\[
\xymatrixcolsep{40pt}
\xymatrix{
A_X(M)\otimes H(Y)\ar[r]^{\id\otimes\alpha_{N,Y}}\ar[d]_{f^*(-)\cap}&A_X(M)\otimes A_Y(N)\ar[d]^{F^*(-)\cup}\\
H(Y\cap f^{-1}(X))\ar[r]_{\alpha_{N,Y\cap f^{-1}(X)}}&A_{Y\cap f^{-1}(X)}(N)
}
\]
commutes.\\
ii. The diagram
\[
\xymatrix{
H(X)\otimes H(Y)\ar[r]^{\times}\ar[d]_{\alpha_{M,X}\otimes\alpha_{N,Y}}&H(X\times Y)\ar[d]^{\alpha_{M\times N,X\times Y}}\\
A_X(M)\otimes A_Y(N)\ar[r]_{\times}&A_{X\times Y}(M\times N)
}
\]
commutes.
\item[(A4)] Let $(M,X)$ be a smooth pair and let $Y\subset X$ be a closed subset. Then the diagram
\[
\xymatrixcolsep{50pt}
\xymatrix{ 
H(X\setminus Y)\ar[r]^{\alpha_{M\setminus Y,X\setminus Y}}\ar[d]_{\del_{X,Y}}&A_{X\setminus Y}(M\setminus Y)\ar[d]^{\del_{M,X,Y}}\\
H(Y)\ar[r]_{\alpha_{M,Y}}&A_Y(M)
}
\]
commutes.
\end{enumerate}
The functor $H$ together with the additional structures (D2)-(D6) is the {\em oriented Borel-Moore homology theory} underlying the oriented duality theory.
\end{definition}

\begin{rem} Oriented duality theories on $\Sch_k$ form a category, in the evident manner. Given a $\Z/2$-graded oriented   ring cohomology theory on $\SP$, an {\em extension} of $A$ to an oriented duality theory on $\Sch_k$ is a oriented duality theory $(H, A')$ together with an isomorphism $A\cong A'$ of  $\Z/2$-graded oriented ring cohomology theories on $\SP$. Clearly, two extensions $(H_1, A_1)$ and $(H_2, A_2)$ of $A$ are uniquely isomorphic as extensions of $A$: the only possible choice of isomorphism $H_1\cong H_2$ compatible with the given isomorphisms $A_1\xrightarrow{\beta} A\xleftarrow{\gamma} A_2$ is given by the isomorphisms
\[
H_1(X)\xrightarrow{\beta\circ \alpha^1_{M,X}}A_X(M)\xleftarrow{\gamma\circ \alpha^2_{M,X}}H_2(X)
\]
for any choice of smooth pair $(M,X)$.
\end{rem}

\begin{rem}\label{rem:Graded2}  One has as well the $\Z$-graded or bi-graded versions of oriented duality theories. For the graded version, one typically uses homological grading on $H$, so that the comparison isomorphisms $\alpha$ are of the form
\[
\alpha_{M,X}:H_n(X)\to A^{2d-n}_X(M)
\]
where $d=\dim_kM$ (by additivity, we may assume that $M$ is equi-dimensional over $k$). Using remark~\ref{rem:Grading1}, the projective push-forward map $f_*$ preserve the grading, as do the pull-back maps for open immersions. The cap products become
\[
A^m_X(M)\otimes H_n(Y)\xrightarrow{f^*(-)\cap}H_{n-m}(Y\cap f^{-1}(X)).
\]

In the bi-graded case, we index $H$ to give comparison isomorphisms
\[
\alpha_{M,X}:H_{p,q}(X)\to A^{2d-p,d-q}_X(M)
\]
The second index in the bi-grading plays the role of the ``weight" in the classical Bloch-Ogus theory. The projective push-forward and open pull-back preserve the bi-grading, and the cap products are 
\[
A^{m,n}_X(M)\otimes H_{p,q}(Y)\xrightarrow{f^*(-)\cap}H_{p-m,q-n}(Y\cap f^{-1}(X)).
\]
\end{rem}

\begin{thm}\label{thm:Main} Suppose that $k$ admits resolution of singularities. Let $A$ be an oriented $\Z/2$-graded ring cohomology theory on $\SP$. Then there is a unique extension of $A$ to an oriented duality theory $(H,A)$ on $\Sch_k$.
\end{thm}

\begin{proof} We have already discussed the uniqueness. The existence follows from the results of \cite[\S2.1]{Mocanasu}, with some minor modifications. We give a sketch of the construction for the reader's convenience. refering to \cite{Mocanasu} for details. We will use throughout theorem~\ref{thm:AlgOrient}, that an oriented $\Z/2$-graded ring cohomology theory defines an oriented algebraic cohomology theory.

 Call morphisms $F,G:(M,X)\to (N,Y)$ in $\SP'$  equivalent if $F$ and $G$ induce the same morphism $X\to Y$, and let $\overline{\SP'}$ be the quotient of $\SP'$ by this equivalence relation.

 We have the restriction functor $\res:\overline{\SP'}\to \Sch'_k$, sending $(M,X)$ to $X$ and $[F]:(M,X)\to (N,Y)$ to the restriction $F_{|X}:X\to Y$. We let $\sH\SP'$ be the category formed from $\overline{\SP'}$ by inverting all morphisms over an isomorphism in $\Sch'_k$.  For each $X$ in $\Sch'_k$, the fiber of $\res$ over $X$ is a left-filtering category with at most one morphism between any two objects, so the induced map $\res:\sH\SP\to \Sch'_k$ is an equivalence of categories.
 
 By definition~\ref{def:AlgOrient}(A2,A3,A7), sending $(M,X)$ to $A_X(M)$ and $F:(M,X)\to (N,Y)$ in $\SP'$ to $F_*:A_X(M)\to A_Y(N)$ descends to a well-defined functor
 \[
 A_{-}(-):\sH\SP'\to \gr_{\Z/2}\Ab.
 \]
 Since $\res:\sH\SP'\to \Sch'_k$ is an equivalence, this gives us the functor $H:\Sch'_k\to \gr_{\Z/2}\Ab$  and the natural isomorphisms
 \[
 \alpha_{M,X}:H(X)\to A_X(M)
 \]
 satisfying axiom (A2). 
 
 To define the pull-back map $j^*:H(X)\to H(Y)$ associated to an open immersion $j:Y\to X$, choose a smooth pair $(M,X)$. It is easy to see that there is a smooth pair $(N,Y)$ and an open immersion $\tilde{j}:N\to M$ extending $j$. Let $j^*:H(X)\to H(Y)$ be the unique map making the diagram
 \[
 \xymatrix{
 H(X)\ar[r]^{\alpha_{M,X}}\ar[d]_{j^*}&A_X(M)\ar[d]^{\tilde{j}^*}\\
 H(Y)\ar[r]_{\alpha_{N,Y}}&A_Y(N)
 }
 \]
 commute.  
 
 To verify (A1), let $(M',X)$, $(N',Y)$ be smooth pairs, and let $g:N'\to   M'$ be an open immersion extending $j$.  We have the commutative diagram
 \[
 \xymatrix{
 Y\ar@{^(->}[dr]\ar@{=}[rr]\ar@{=}[dd]&&Y\ar@{^(->}[d]\ar@/^20pt/[dddr]^j
 \\
 &N'\times N\ar[r]^{\id\times \tilde{j}}\ar[d]_{g\times\id}&N'\times M\ar[d]^{g\times\id}\\
 Y\ar@{^(->}[r]\ar@/_20pt/[drrr]_j&M'\times N\ar[r]_{\id\times\tilde{j}}&M'\times M\\
&&&X\ar@{^(->}[ul]
}
\]
By definition~\ref{def:AlgOrient}(A5), we have
\[
((g\times\id)_*)^{-1}\circ (\id\times\tilde{j})^*=((\id\times\tilde{j}_*)^{-1}\circ (g\times\id)^*
\]
From the commutative diagram
\[
\xymatrix{
Y\ar@{^(->}[r]\ar[d]_j&N'\times M\ar[d]_{g\times\id}\ar[r]^{p_1}&N'\ar[d]^g\\
X\ar@{^(->}[r]&M'\times M\ar[r]^{p_1}&M'
}
\]
and definition~\ref{def:AlgOrient}(A4), we have
\[
p_{1*}\circ(g\times\id)^*=g^*\circ p_{1*}.
\]
Similarly, the commutative diagram
\[
\xymatrix{
Y\ar@{^(->}[r]\ar[d]_j&M'\times N\ar[d]_{g\times\id}\ar[r]^{p_2}&N\ar[d]^g\\
X\ar@{^(->}[r]&M'\times M\ar[r]^{p_2}&M
}
\]
gives
\[
p_{2*}\circ(\id\times\tilde{j})^*=\tilde{j}^*\circ p_{2*}.
\]
Thus, we have the commutative diagrams
\[
\xymatrix{
H(Y)\ar[rd]_{\alpha_{M'\times N,Y}}\ar@/^20pt/[rrd]^{\alpha_{N,Y}}\\
&A_Y(M'\times N)\ar[r]_-{p_{2*}}&A_Y(N)\\
&A_X(M'\times M)\ar[u]^{(\id\times\tilde{j})^*}\ar[r]^-{p_{2*}}&A_X(M)\ar[u]_{\tilde{j}^*}\\
H(X)\ar[ru]^{\alpha_{M'\times M,X}}\ar@/_20pt/[rru]_{\alpha_{M,X}}\ar[uuu]^{j^*}
}
\]
\[
 \xymatrix{
 H(Y)\ar[dr]^{\alpha_{N'\times N,Y}}\ar@{=}[rr]\ar@{=}[dd]&&H(Y)\ar[d]_{\alpha_{N'\times M,Y}}
 \\
 &A_Y(N'\times N)\ar[r]^{(\id\times \tilde{j})_*}\ar[d]_{(g\times\id)_*}&A_Y(N'\times M)\\
H(Y)\ar[r]_{\alpha_{M'\times N,Y}}&A_Y(M'\times N)
&A_Y(M'\times M)\ar[u]_{(g\times\id)^*}\ar[l]^{(\id\times\tilde{j})^*}\\
&&&H(X)\ar[ul]_{\alpha_{M'\times M,X}}\ar@/^20pt/[ulll]_{j^*}\ar@/_20pt/[uuul]_{j^*}
}
\]
and
\[
\xymatrix{
H(Y)\ar[rd]_{\alpha_{N'\times M,Y}}\ar@/^20pt/[rrd]^{\alpha_{N',Y}}\\
&A_Y(M'\times N)\ar[r]_-{p_{1*}}&A_Y(N')\\
&A_X(M'\times M)\ar[u]^{(g\times\id)^*}\ar[r]^-{p_{1*}}&A_X(M')\ar[u]_{g^*}\\
H(X)\ar[ru]^{\alpha_{M'\times M,X}}\ar@/_20pt/[rru]_{\alpha_{M',X}} 
}
\]
 
Thus, the diagram
\[
\xymatrix{
H(Y)\ar[r]^{\alpha_{N',Y}} &A_Y(N')\\
 H(X)\ar[r]_{\alpha_{M',X}} \ar[u]^{j^*}&A_X(M')\ar[u]_{g^*}\\
}
\]
commutes, as desired.

To define the cap product pairing (D4)(i) for a smooth pair $(M,X)$ and a morphism $f:Y\to M$ wth $f(Y)\subset X$,  choose a smooth pair $(N,i:Y\to N)$ and embed $Y$ in $N\times M$ by $(i,f)$. Let $f^*(-)\cap$ be the unique morphism making
\[
\xymatrixcolsep{60pt}
\xymatrix{
A_X(M)\otimes H(Y)\ar[r]^{\id\otimes\alpha_{N\times M,Y}}\ar[d]_{f^*(-)\cap}&A_X(M)\otimes A_Y(N\times M)\ar[d]^{p_2^*(-)\cup}\\
H(Y\cap f^{-1}(X))\ar[r]_{\alpha_{N\times M,Y\cap f^{-1}(X)}}&A_{Y\cap p_2^{-1}(X)}(N\times M)
}
\]
commute. If we have another smooth pair $(N', i':Y\to N')$ and morphism $G:N'\to M$ extending $f$, consider the commutative diagram
\[
\xymatrix{
N'\ar[dr]_{G}\ar[r]^-{(i',G)}&N'\times M  \ar[d]^{G\circ p_2}\\
&M}
\]
We embed $Y$ in $N'\times M$ by $(i',f)$. The projection formula gives, for $b\in A_Y(N')$, $a\in A_X(M)$,
\[
(i',G)_*(G^*(a)\cup b)=(i',G)_*((i',G)^*p_2^*(a)\cup b)=p_2^*(a)\cup (i',G)_*(b),
\]
so we can replace $(i':Y\to N', G:N'\to M)$ with $((i',f):Y\to N'\times M,p_2)$. Similarly, we have the embedding $(i',i,f):Y\to N'\times N\times M$ and for $a\in A_X(M)$, $b\in A_Y(N'\times N\times M)$, we have
\begin{align*}
&p_{N'M*}^{N'NM}(p_M^{N'NM*}(a)\cup b)=p_M^{N'M*}(a)\cup p_{N'M*}^{N'NM}(b),\\
&p_{NM*}^{N'NM}(p_M^{N'NM*}(a)\cup b)= p_M^{NM*}(a)\cup p_{NM*}^{N'NM}(b).
\end{align*}
Here $p_{NM}^{N'NM}$ is the projection $N'\times N\times M\to N\times M$, etc. The commutativity in (A3) follows from these identitites.

For the external product (D4)(ii), we fix as above smooth pairs $(M,X)$, $(N,Y)$ and define $\times:H(X)\otimes H(Y)\to H(X\times Y)$ as the unique map making 
\[
\xymatrix{
H(X)\otimes H(Y)\ar[r]^{\times}\ar[d]_{\alpha_{M,X}\otimes\alpha_{N,Y}}&H(X\times Y)\ar[d]^{\alpha_{M\times N,X\times Y}}\\
A_X(M)\otimes A_Y(N)\ar[r]_{\times}&A_{X\times Y}(M\times N)
}
\]
commute. If we have other smooth pairs $(M',X)$, $(N',Y)$, consider the diagram 
\[
\xymatrix{
A_X(M\times M')\otimes A_Y(N\times N')\ar[d]_{p_{1*}\otimes p_{1*}}\ar[r]^\times&
A_{X\times Y}(M\times M'\times N\times N')\ar[d]^{p_{13*}}\\
A_X(M)\otimes A_Y(N)\ar[r]_\times& A_{X\times Y}(M\times N)
}
\]
By remark~\ref{rem:ProjForm}, this diagram commutes.  Using the similar diagram with $M',N'$ replacing $M,N$ in the bottom row verifies (A3)(ii).

Finally, for (A4), choose a smooth pair $(M,X)$, and let $\del_{X,Y}$ be the unique map making
\[
\xymatrixcolsep{50pt}
\xymatrix{
H(X\setminus Y)\ar[r]^{\alpha_{M\setminus Y,X\setminus Y}}\ar[d]_{\del_{X,Y}}&A_{X\setminus Y}(M\setminus Y)\ar[d]^{\del_{M,X,Y}}\\
H(Y)\ar[r]_{\alpha_{M,Y}}&A_Y(M)
}
\]
commute. If we have another smooth pair $(M',X)$, we have as well the smooth pair $(M\times M',X)$ and commutative diagram
\[
\xymatrix{
A_{X\setminus Y}(M\setminus Y)\ar[d]_{\del_{M,X,Y}}&\ar[l]_-{p_{1*}}\ar[r]^-{p_{2*}}
A_{X\setminus Y}(M\times M'\setminus Y)\ar[d]^{\del_{M\times M',X,Y}}&
A_{X\setminus Y}(M'\setminus Y)\ar[d]^{\del_{M',X,Y}}\\
A_Y(M)&A_Y(M\times M')\ar[l]_-{p_{1*}}\ar[r]^-{p_{2*}}&A_Y(M').
}
\]
(see lemma~\ref{lem:BoundaryExtension}) from which (A4) follows directly.

\end{proof}

Of course, the role of the Borel-Moore homology theory $H$ in an oriented duality theory is just to say that certain properties of cohomology with supports $A_X(M)$ depend only on $X$, not on the choice of smooth pair $(M,X)$. Besides the properties given by the axiomatics (projective push-forward, open pull-back, cup products and boundary map) one has the following properties and structures:\\
\\
{\em Functoriality of open pull-back, cap products and external products}.  For $j:U\to V$, $g:V\to X$ open immersions in $\Sch_k$, we have
\[
j^*\circ g^*=(g\circ j)^*:H(X)\to H(U)
\]
and $\id_X^*=\id_{H(X)}$.
This follows from the functoriality of   open pull-back for the oriented ring cohomology theory $A$, using (A1) to compare.

For the cap products, we have three functorialities:
\begin{enumerate}
\item Take $X, Y\in\Sch_k$, $M\in \Sm/k$, a smooth pair $(N,X)$ and morphisms $f:Y\to M$, $g:M\to N$, . Then
\[
(g\circ f)^*(a)\cap b=f^*(g^*(a))\cap b
\]
for $a\in A_X(N)$, $b\in H(Y)$, with $g^*:A_X(N)\to A_{g^{-1}(X)}(M)$ the pull-back. 
\item Let $h:Y\to Z$ be a projective morphism in $\Sch_k$, $(M,X)$ a smooth pair, and $f:Z\to M$ a morphism. Then
\[
h_*((f\circ h)^*(a)\cap b)=f^*(a)\cap h_*(b).
\]
\item Let $j:U\to Y$ be an open immersion in $\Sch_k$, and let $(M,X)$ be a smooth pair, and let $f:Y\to M$ be a morphism. Then
\[
j^*(f^*(a)\cap b)=((f\circ j)^*(a)\cap j^*(b)
\]
for $a\in A_X(M)$, $b\in H(Y)$.
\end{enumerate}
The first and third identities  follow from the naturality of cup product with respect to pull-back, and the second from the projection formula. Finally, the fact that pull-back is a ring homomorhism yields the identity
\[
f^*(a\cup b)\cap c=f^*(a)\cap(f^*(b)\cap c)
\]
for $a,b\in A_X(M)$, $c\in H(Y)$.

The external products are functorial for push-forward: For projective morphisms $f:X\to X'$, $g:Y\to Y'$, we have
\[
(f\times g)_*(a\times b)=f_*(a)\times g_*(b)\in H(X'\times Y');\quad a\in H(X), b\in H(Y).
\]
This follows from remark~\ref{rem:ProjForm}.
\\
{\em Long exact sequence of a pair and Mayer-Vietoris}. Let $i:Y\to X$ be a closed subset of $X\in \Sch_k$ and let $j:U\to X$ be the open complement. Then the sequence
\[
\ldots\to H(U)\xrightarrow{\del_{X,Y}}H(Y)\xrightarrow{i_*}H(X)\xrightarrow{j^*}H(U)\to\ldots
\]
is exact. Indeed, we  use (A1), (A2) and (A4) to compare with the long exact sequence of the triple $(M,X,Y)$, having chosen a smooth pair $(M,X)$.

If we have an open cover of some $X\in\Sch_k$, $X=U\cup V$, the exact sequence of a pair gives formally the long exact Mayer-Vietoris sequence
\[
\ldots\to H(X)\xrightarrow{(j_U^*,j_V^*)}H(U)\oplus H(V)\xrightarrow{j^{U*}_{UV}-j^{V*}_{UV}}
H(U\cap V)\xrightarrow{\del_{X,U,V}}H(U\cap V)\to\ldots
\]

The localization sequence is natural with respect to pull-back by open immersions and by push-forward with respect to projective morphisms. 
\begin{prop} Let $(H, A)$ be an oriented duality theory.\\
\\
1. Let $Y\subset Y'\subset X$ be closed subsets of $X\in\Sch_k$. Let $U=X\setminus Y$, $U'=X\setminus Y'$, with inclusions $j:U'\to U$ and $i:Y\to Y'$. Then the diagram
\[
\xymatrix{
H_{a,b}(U)\ar[r]^{\partial_{X,Y}}\ar[d]_{j^*}&H_{a-1,b}(Y)\ar[d]^{i_*}\\
H_{a,b}(U')\ar[r]_{\partial_{X,Y'}}&H_{a-1,b}(Y')
}
\]
commutes.\\
\\
2. Let $f:X'\to X$ be a projective morphism in $\Sch_k$, let $Y\subset X$ be a closed subset,  let $Y'=f^{-1}(Y)$, $U=X\setminus Y$, $U'=X'\setminus Y'$, and let $f_U:U'\to U$, $f_Y:Y'\to Y$  be the respective restrictions of $f$. Then $f_U$ is projective and the diagram
\[
\xymatrix{
H_{a,b}(U')\ar[d]_{f_{U*}}\ar[r]^{\partial_{X',Y'}}&H_{a-1,b}(Y')\ar[d]^{f_{Y*}}\\
H_{a,b}(U)\ar[r]_{\partial_{X,Y}}&H_{a-1,b}(Y)
}
\]
commutes.
\end{prop}

\begin{proof} For (1), take a closed immersion $X\to M$ with $M\in \Sm/k$. Let $N=M\setminus Y$, $N'=M\setminus Y'$. The identity map on $M$ gives the map in $\SP$,  $(M, Y')\to (M,Y)$, and the inclusion $N'\to N$ gives the map $(N,U)\to (N',U')$; these arise from the map of triples $(M,X,Y)\to (M,X,Y')$.   Via the comparison isomorphisms $\alpha_{**}$, the diagram in (1) is isomorphic to
\[
\xymatrix{
E^{p,q}_U(N)\ar[r]^{\partial_{M,X,Y}}\ar[d]_{j^*}&E^{p+1,q}_Y(M)\ar[d]^{\id^*}\\
E^{p,q}_{U'}(N')\ar[r]_{\partial_{X,Y'}}&E^{p+1,q}_{Y'}(M)
}
\]
The commutativity of this diagram follows directly from the naturality of $\del_{**}$ (definition~\ref{def:cohomology}(1)) and the construction of the long exact sequence of a triple.

A similar argument proves (2). Indeed, take a closed immersion $X\to M$ with $M\in\Sm/k$. Since $f:X'\to X$ is projective, we can factor $f$ as a closed immersion $i:X'\to X\times\P^n$ followed by the projection $X\times\P^n\to X$. This gives us the closed immersion $X'\to M\times\P^n$ and the projection $M\times\P^n\to M$ extends $f$, giving us the map $(M\times\P^n,X')\to M,X)$ in $\SP'$. Using the naturality of $\del_{**}$ described in definition~\ref{def:IntSupp}(6) completes the proof.
 \end{proof}
\ \\
{\em Pull-back by a smooth projection}. Although it appears that smooth pull-back depends on the choice of smooth pair, one does have a well defined pull-back 
\[
p^*:H(X)\to H(X\times F)
\]
for $F\in \Sm/k$.

\begin{lem} For $F\in\Sm/k$, $X\in\Sch_k$, let $p:X\times F\to F$ be the projection. Then there is a pull-back map $p^*:H(X)\to H(X\times F)$ such that, for each smooth pair $(M,X)$, the diagram
\[
\xymatrixcolsep{50pt}
\xymatrix{
H(X)\ar[r]^{\alpha_{M,X}}\ar[d]_{p^*}&A_X(M)\ar[d]^{p_*}\\
H(X\times F)\ar[r]_{\alpha_{M\times F,X\times F}}&A_{X\times F}(M\times F)
}
\]
commutes.
\end{lem}

\begin{proof} Of course, we define $p^*:H(X)\to H(X\times F)$ to be the unique map making the above diagram commute, for one fixed choice $(M,X)$ of a smooth pair.

Let  $(N,X)$ be another smooth pair. We have the cartesian transverse diagram
\[
\xymatrix{
N\times F\ar[d]_q&M\times N\times F\ar[l]_-{\pi_{23}}\ar[r]^-{\pi_{13}}\ar[d]_{\pi_{12}}
&M\times F\ar[d]^p\\
N&M\times N\ar[l]^-{p_2}\ar[r]_-{p_1}&M
}
\]
where the maps are the respective projections. This gives us the commutative diagram
\[
\xymatrix{
A_{X\times F}(N\times F)&\ar[l]_-{\pi_{23*}}\ar[r]^-{\pi_{13*}}A_{X\times F}(M\times N\times F)
&A_{X\times F}(M\times F)\\
A_X(N)\ar[u]^{q^*}&A_X(M\times N)\ar[u]_{\pi_{12}^*}\ar[l]^-{p_{2*}}\ar[r]_-{p_{1*}}&A_X(M)
\ar[u]_{p^*}
}
\]
which gives the desired commutativity.
\end{proof}
The cap product is also natural with respect to this pull-back, and we have
\[
p_U^*\circ j^*=(j\times\id)^*\circ p^*
\]
for an open immersion $j:U\to X$, where $p_U:U\times F\to U$ is the projection. Finally, for $g:V\to F$ an open immersion in $\Sm/k$, let $p_V:X\times V\to X$ be the projection. Then
\[
p_V^*=(\id\times g)^*\circ p^*.
\]
\ \\
{\em Homotopy invariance}. Let $p:\A^n\times X\to X$ be the projection. Then 
\[
p^*:H(X)\to H(\A^n\times X)
\]
is an isomorphism. This follows directly from the homotopy invariance of $A$, together with the existence of the well-defined pull-back $p^*$.\\
\\
{\em Chern class operators}. Let $E\to X$ be a vector bundle of rank $r$ on some $X\in \Sch_k$. $X$ is quasi-projective,  so choose a closed immersion $i:X\to U$, with $U\subset \P^n$ an open subscheme. This gives us  the  very ample line bundle $O_X(1)$ on $X$. For $m>>0$, the vector bundle $E(m)$ is generated by global sections; a choice of generating sections $s_0,\ldots, s_M$ gives a morphism $f:X\to \Gr(M,r)$ with $f^*(E_{M,r})\cong E(m)$, where $E_{M,r}\to \Gr(M,r)$ is the universal bundle. Thus, we have the locally closed immersion $(i,f):X\to \P^n\times\Gr(M,r)$ with $(i,f)^*(\sO(-m)\boxtimes E_{M,r})\cong E$; choosing an open subscheme $V\subset \P^n\times\Gr(M,r)$ such that $(i,f):X\to V$ is a closed immersion, we have a smooth pair $(V,X)$ and a vector bundle $\sE$ on $V$ which restricts to $E$ on $X$. Define the {\em Chern class operator}
\[
\tilde{c}_p(L):H(X)\to H(X)
\]
by setting $\tilde{c}_p(L)(b):=(i,f)^*(c_p(\sE))\cap b$. 

One needs to check that $\tilde{c}_p(E)$ is independent of the choices we have made.  This follows from
\begin{prop}[\hbox{\cite[\S3.2, Lemma]{Fulton}}] \label{prop:FultonK0} For $X\in Sch_k$, the pull-back of locally free sheaves induces an isomorphism
\[
K_0(X) \to \lim_{\substack{\to\\f:X\to V\in\Sm/k}}K_0(V).
\]
\end{prop}
Now, suppose we have two smooth pairs $(M,X)$ and $(N,X)$, with vector bundles $E_M$ on $M$, $E_N$ on $N$, restricting to $E$ on $X$. By the proposition, there is a $V\in \Sm/k$, a vector bundle $E_V$ on $V$ and a commutative diagram
\[
\xymatrix{
&N\\
X\ar@{^(->}[ur]^{i_N}\ar@{^(->}[r]^{i_V}\ar@{^(->}[dr]_{i_M}&V\ar[u]_f\ar[d]^g\\
&M
}
\]
such that $[E_V]=[f^*E_N]=[g^*E_M]\in K_0(V)$. In particular, this implies that $c_p(E_V)=f^*(c_p(E_N))=g^*(c_p(E_M))$ in $A(V)$, and thus 
\[
i_M^*(c_p(E_M))\cap(-)=i_N^*(c_p(E_N))\cap(-):H(Y)\to H(Y).
\]

Proposition~\ref{prop:FultonK0} gives 
\begin{lem} Let $E, E'$ be vector bundles on $X\in\Sch_k$. Then for all $p,q$, the Chern class operators $\tilde{c}_p(E), \tilde{c}_q(E')$ commute. For $p\ge1$, $\tilde{c}_p(E)$ is nilpotent, $\tilde{c}_0(E)$ is the identity operator and $\tilde{c}_p(E)=0$ for $p>\text{rank }E$.
\end{lem}
Indeed, these properties  for the Chern classes $c_p(E)\in A(V)$, $V\in\Sm/k$, follow from \cite[Theorem 3.6.2]{Panin2}.

Similarly, one has the Whitney product formula for the total Chern class operator. Let $\tilde{c}(E)= \sum_{p=0}^{\text{rank }E}\tilde{c}_p(E)$.

\begin{lem}  Let $0\to E'\to E\to E''\to0$ be an exact sequence of vector bundles on $X\in\Sch_k$. Then
\[
\tilde{c}(E)=\tilde{c}(E')\circ\tilde{c}(E'')=\tilde{c}(E'')\circ\tilde{c}(E').
\]
\end{lem}

Indeed, this follows from proposition~\ref{prop:FultonK0} plus the Whitney product formula for the total Chern class $c(E):=\sum_{p=0}^{\text{rank }E}c_p(E)\in A^{ev}(V)$,  for $E\to V$ a vector bundle, $V\in \Sm/k$ (see \cite[Theorem 3.6.2]{Panin2}).

The same reasoning shows that the formal group law for $A$ extends to $H$:
\begin{lem} Let $L, M$ be line bundles on $X\in\Sch_k$. Then 
\[
F_A(\tilde{c}_1(L), \tilde{c}_1(M))=\tilde{c}_1(L\otimes M).
\]
\end{lem}

The properties of the cap product with respect to pull-back and push-forward give
\begin{enumerate}
\item Let $f:Y\to X$ be a projective morphism in $\Sch_k$, $E\to X$ a vector bundle. Then 
\[
f_*\circ \tilde{c}_p(f^*E)=\tilde{c}_p(E)\circ f_*
\]
\item Let $j:U\to X$ be an open immersion, $p:X\times F\to X$ a projection, with $F\in\Sm/k$. Then for $E\to X$ a vector bundle, we have
\[
j^*\circ \tilde{c}(E)=\tilde{c}(j^*E)\circ j^*;\quad p^*\circ \tilde{c}(E)=\tilde{c}(p^*E)\circ p^*.
\]
\end{enumerate}
Finally, the projective bundle formula with supports (remark~\ref{rem:ProjBundleFormSupp}) and the cap products give the projective bundle formula for $H$: For $X\in\Sch_k$, let
 $p:\P^n\times X\to X$ the projection, and let 
 \[
 \alpha_i:H(X)\to H(\P^n\times X)
 \]
 be the composition $\tilde{c}_1(O(1))^i\circ p^*$. Then
 \[
 \sum_{i=0}^n\alpha_i:H(X)^{n+1}\to H(\P^n\times X)
 \]
 is an isomorphism.

\section{Algebraic cobordism} We want to consider the two varieties of algebraic cobordism: the bi-graded theory $\MGL^{*,*}$ represented by the algebraic Thom complex $\MGL\in\SH(k)$, and the theory $\Omega_*$, the universal oriented Borel-Moore homology theory on $\Sch_k$ (in the sense of \cite[Definition 5.1.2]{LevineMorel}). As above, we will assume that $k$ admits resolution of singularities.
For the basic definitions and notions of motivic homotopy theory used below, we refer the reader to \cite{Motivic, MorelVoev, MorelLec, Voevodsky}.

The Thom complex $\MGL$ is constructed from the Thom spaces of the universal bundles $E_n\to \BGL_n$, $\MGL_n:=Th(E_n):=E_n/(E_n\setminus 0_{\BGL_n})$,
\[
\MGL:=(pt,\MGL_1,\MGL_2,\ldots).
\]
The bonding maps are given via the inclusions $i_n:\BGL_n\to \BGL_{n+1}$, noting that $i_n^*(E_{n+1})\cong E_n\oplus O_{\BGL_n}$, and thus we have
\[
\Sigma_tTh(E_n)\cong Th(E_n\oplus O_{\BGL_n})\cong Th(i_n^*E_{n+1})\xrightarrow{\tilde{i}_n}
Th(E_{n+1}).
\]

We recall from \cite[3.8.7]{Panin2} the orientation on  $\MGL^{*,*}$. First of all,  $\MGL^{*,*}$ is a bi-graded ring cohomology theory on $\SP$, with 
\[
\MGL^{p,q}_X(M):=\Hom_{\SH(k)}(\Sigma^\infty_tM/(M\setminus X),\Sigma^{p,q}\MGL).
\]
The ring structure is given by the canonical lifting of $\MGL$ to a ring object in the category of symmetric $T$-spectra (see e.g. \cite{PaninPimenovRoendigs}). The orientation is given by a Thom structure and Panin's theorem \cite[Theorem 3.7.4]{Panin2}, which associates an orientation to a ring cohomology theory with a Thom structure. The Thom structure is induced by choosing a Thom class on the universal Thom space $Th(O_{\P^\infty}(1))$, which we now describe. Since $\P^\infty=\BGL_1$, and $O_{\P^\infty}(1)$ is the universal bundle on $\BGL_1$, the Thom space 
\[
Th(O_{\P^\infty}(1)):=O_{\P^\infty}(1)/(O_{\P^\infty}(1)\setminus \P^\infty)
\]
 is by definition equal to $\MGL_1$. The identity map on $Th(O_{\P^\infty}(1))$ thus extends canonically to a map
\[
\iota:\Sigma^\infty_t Th(O_{\P^\infty}(1))\to \Sigma_t\MGL=\Sigma^{2,1}\MGL
\]
giving the universal Thom class $[\iota]\in \MGL^{2,1}_{\P^\infty}(O_{\P^\infty}(1))$. If now $L\to M$ is a line bundle on some $M\in \Sm/k$, Jouanoulou's trick gives us an affine space bundle $p:M'\to M$ with $M'$ affine. We replace $L\to M$ with $L'\to M'$, giving the $\A^1$ weak equivalence $Th(\tilde{p}):Th(L')\to Th(L)$, and thus,  the isomorphism
\[
Th(\tilde{p})^*:\MGL^{*,*}_M(L)\to \MGL^{*,*}_{M'}(L').
\]
As $M'$ is affine,  $L'$ is generated by global sections, so there is a morphism $f:M'\to  \P^\infty$ with $L'\cong f^*(O_{\P^\infty}(1))$. One defines 
\[
th(L)\in \MGL^{2,1}_M(L)  
\]
as $th(L)=(Th(\tilde{p})^*)^{-1}\circ f^*([\iota])$.

\begin{prop}\label{prop:MGL} Let $k$ be a field admitting resolution of singularities. Then there is a unique bi-graded oriented duality theory $(\MGL'_{*,*},\MGL^{*,*})$ such that the orientation on $\MGL^{*,*}$ is the one with associated Thom structure given by the universal Thom class $[\iota]\in \MGL^{2,1}_{\P^\infty}(O_{\P^\infty}(1))$.
\end{prop}
We use the notation $\MGL'_{*,*}$ to distinguish the Borel-Moore homology theory from the homology theory
\[
\MGL_{p,q}(X):=\Hom_{\SH(k)}(S^{p,q}_k,\MGL\wedge \Sigma^\infty_tX_+).
\]

\begin{proof}
Indeed, the Thom class assignment  $L\mapsto th(L)\in \MGL^{*,*}_M(L)$ described above is shown to give a Thom structure on $\MGL^{*,*}$ in e.g. \cite{PaninPimenovRoendigs}. Panin's theorem \cite[Theorem 3.7.4]{Panin2} gives the associated orientation for $\MGL^{*,*}$, and we may apply the bi-graded version of theorem~\ref{thm:Main} to complete the proof.
\end{proof}

We now turn to the ``geometric" theory $\Omega_*$. In spite of the terminology, $\Omega_*$ does not satisfy all the properties of the underlying Borel-Moore homology theory of a $\Z$-graded oriented duality theory: instead of the long exact sequence of a pair $i:Y\to X$, one has a  right-exact sequence
\[
\Omega_n(Y)\xrightarrow{i_*}\Omega_n(X)\xrightarrow{j^*}\Omega_n(X\setminus Y)\to 0.
\]

In any case,  $\Omega_*$ does act as if it were at least part of a universal theory. Given a functor
\[
H_{*,*}:\Sch'_k\to \bgr\Ab
\]
and an $X\in\Sch_k$, we let $H_{2*,*}(X)=\oplus_nH_{2n,n}(X)$, giving the functor
\[
H_{2*,*}:\Sch'_k\to\gr\Ab.
\]

\begin{prop} \label{prop:OmegaUniv} Let $k$ be a field admitting resolution of singularities, and let $(H,A)$ be a bi-graded oriented duality theory. Then there is a unique natural transformation
\[
\vartheta_H:\Omega_*\to H_{2*,*}
\]
of functors $\Sch'_k\to\gr\Ab$, satisfying
\begin{enumerate}
\item Let $j:U\to X$ be an open immersion in $\Sch_k$. Then the diagram
\[
\xymatrixcolsep{35pt}
\xymatrix{
\Omega_*(X)\ar[r]^{\vartheta_H(X)}\ar[d]_{j^*}&H_{2*,*}(X)\ar[d]^{j^*}\\
\Omega_*(U)\ar[r]_{\vartheta_H(X)}&H_{2*,*}(U)
}
\]
commutes.
\item Let $f:M\to N$ be a morphism in $\Sm/k$, $d_N=\dim_kN$, $d_M=\dim_kM$, $d=\codim f:=d_N-d_M$. Then the diagram
\[
\xymatrixcolsep{50pt}
\xymatrix{
\Omega_*(N)\ar[d]_{f^*}\ar[r]^-{\alpha_{N,N}\circ\vartheta_H(N)}&A^{2d_N-2*, -*}(N)\ar[d]^{f^*}\\
\Omega_{*-d}(M)\ar[r]_-{\alpha_{M,M}\circ\vartheta_H(M)}&A^{2d_N-2*, -*}(M)
}
\]
commutes.
\item Let $L\to X$ be a line bundle on some $X\in\Sch_k$. Then the diagram
\[
\xymatrixcolsep{35pt}
\xymatrix{
\Omega_*(X)\ar[d]_{\tilde{c}_1(L)}\ar[r]^{\vartheta_H(X)}&H_{2*,*}(X)\ar[d]^{\tilde{c}_1(L)}\\
\Omega_{*-1}(X)\ar[r]_{\vartheta_H(X)}&H_{2*-2,*-1}(X)
}
\]
commutes.
\item For $M\in\Sm/k$ of dimension $d_M$ over $k$, set $\Omega^n(M):=\Omega_{d_M-n}(M)$. Then the map $\alpha_{M,M}\circ\vartheta_H:\Omega^*(M)\to A^{2*,*}(M)$ is a homomorphism of graded rings.
\end{enumerate}
\end{prop}

\begin{proof} We let $\L_*$ denote the Lazard ring, that is, the coefficient ring of the universal rank one commutative formal group law, 
\[
F_\L(u,v):=u+v+\sum_{i,j\ge1}a_{ij}u^iv^j
\]
$\L$ is generated as a commutative $\Z$-algebra by the coefficients $s_{ij}$, and we give $\L$ the grading with $\deg(a_{ij})=i+j-1$. We use the construction of $\Omega_*$ as the universal 
``oriented Borel-Moore functor of geometric type" on $\Sch_k$ (see \cite[Definitions 2,1,1, 2.1.12, 2.2.1 and Theorem 2.3.13]{LevineMorel}). Since $H_{2*,*}$ does not have all the properties of an oriented Borel-Moore functor of geometric type, we are forced to go through the actual construction of $\Omega_*$; we will give a sketch of this three-step process, referring the reader to \cite[\S2]{LevineMorel} for the details.

\noindent
{\bf Step 1}. For $Y\in\Sm/k$ of dimension $d_Y$ over $k$, let $p_Y:Y\to pt$ be the structure morphism, and define the {\em fundamental class} $[Y]_H\in H_{2d_Y,d_Y}(Y)$ by
\[
[Y]_H=\alpha_{Y,Y}^{-1}(p_Y^*(1))
\]
where $1\in A^{0.0}(pt)$ is the unit.

For $X\in\Sch_k$, let $\sZ_n(X)$ denote the group of {\em dimension $n$ cobordism cycles} on $X$. This is the group generated by tuples $(f:Y\to X;L_1,\ldots, L_r)$, with $Y\in\Sm/k$ irreducible of dimension $n+r$ over $k$, $f$ a projective morphism, and $L_1,\ldots, L_r$ line bundles on $Y$ (we allow $r=0$). We identify two cobordism cycles by isomorphism over $X$ (see \cite[Definition 2.1.6]{LevineMorel}. Note that this includes reordering the $L_i$).  $\sZ_*(X)$ has the following operations:
\begin{enumerate}
\item[i.] {\em Projective push-forward}. For $f:X\to X'$ a projective map in $\Sch_k$, set
\[
g_*((f:Y\to X;L_1,\ldots, L_r)):=(g\circ f:Y\to X';L_1,\ldots, L_r).
\]
\item[ii.] {\em Smooth pull-back}. Let $h:X'\to X$ be a smooth, quasi-projective morphism of relative dimension $d$. Set
\begin{multline*}
\hskip 30pt h^*((f:Y\to X;L_1,\ldots, L_r))\\
:=(p_2:Y\times_XX'\to X';p_1^*L_1,\ldots, p_1^*L_r)\in \sZ_{d+n}(X').
\end{multline*}
\item[iii.] {\em Chern class operator}. Let $L\to X$ be a line bundle. Set
\begin{multline*}
\hskip 30pt \tilde{c}_1(L)((f:Y\to X;L_1,\ldots, L_r))\\:=(f:Y\to X;L_1,\ldots, L_r, f^*L)\in\sZ_{n-1}(X).
\end{multline*}
\item[iv.] {\em External products}. Define 
\[
\times:\sZ_n(X)\times\sZ_m(X')\to \sZ_{n+m}(X\times X')
\]
by 
\begin{multline*}
\hskip 30pt (f:Y\to X;L_1,\ldots, L_r)\times (f':Y'\to X';M_1,\ldots, M_s)\\=
(f\times f':Y\times Y'\to X\times X';p_1^*L_1,\ldots, p_1^*L_r, p_2^*M_1,\ldots, p_2^*M_s)
\end{multline*}
\end{enumerate}
Define $\vartheta_H(X):\sZ_*(X)\to H_{2*,*}(X)$ by
\[
\vartheta_H(X)((f:Y\to X;L_1,\ldots, L_r)):=f_*(\tilde{c}_1(L_1)\circ\ldots\circ\tilde{c}_1(L_r)([Y]_H))\in H_{2d_Y-2r,d_Y-r}(X).
\]
The properties of projective push-forward, pull-back for open immersions, and Chern class operators for $H$ that we have discussed in \S\ref{sec:OrientDualityThy} imply that the $\vartheta_H(X)$ define a natural transformation of functors
\[
[\vartheta^1_H:\sZ_*\to H_{2*,*}]:\Sch_k'\to\Gr\Ab,
\]
and that $\vartheta^1_H$ is compatible with pull-back for open immersions, and with the respective Chern class operators for line bundles. 

For $M\in \Sm/k$ of dimension $d_M$ over $k$, we let $\sZ^*(M):=\sZ_{d_M-*}(M)$. We have the map
\[
\vartheta_1^A(M):=\alpha_{M,M}\circ\vartheta^1_H(M):\sZ^*(M)\to A^{2*,*}(M).
\]
It is easy to see that $\vartheta_1^A$ has the same compatibilities as $\vartheta_H$, and in addition, $\vartheta^A$ is compatible with smooth pull-back and external products. \\
\\
{\bf Step 2}. The formal group law $F_A(u,v)\in A^{2*,*}(pt)[[u,v]]$ gives rise to the classifying map
\[
\phi_A:\L_*\to A^{-2*,-*}(pt),
\]
a homomorphism of graded rings. Via the structure morphism $p_X:X\to pt$, $H_{2*,*}(X)$ becomes a graded module over $A^{-2*,-*}(pt)$, and the projective push-forward, open pull-back and Chern class operators are all $A^{-2*,-*}(pt)$-module maps. Via $\phi_A$, $H_{2*,*}(X)$ becomes a 
graded module over $\L_*$, and  the projective push-forward, open pull-back and Chern class operators are all $\L_*$-module maps. Thus, $\vartheta^1_H$ gives rise to the natural transformation
\[
[\vartheta^2_H:\L_*\otimes\sZ_*\to H_{2*,*}]:\Sch_k'\to \gr_{\L_*}\Mod
\]
compatible with open pull-back and Chern class operators.

Similarly, we have the maps
\[
\vartheta_2^A(M):\L^*\otimes\sZ^*(M)\to A^{2*,*}(M),
\]
compatible with projective push-forward, smooth pull-back, Chern class operators and external products. Here $\L^n:=\L_{-n}$, giving the graded ring $\L^*$ and the graded ring homomorphism $\phi_A:\L^*\to A^{2*,*}(pt)$. The graded group $A^{2*,*}(M)$ is thereby a graded $\L^*$-module, and the 
projective push-forward, smooth pull-back and Chern class operators are all $\L^*$-linear. The external products are $\L^*$-bilinear.\\
\\
{\bf Step 3}. The group $\Omega_*(X)$ is defined as a quotient of $\L_*\otimes\sZ_*(X)$ by imposing relations on the Borel-Moore functor $\L_*\otimes\sZ_*$ (see \cite[Definition 2.2.1,]{LevineMorel}):
\begin{enumerate}
\item[a.] {\em The dimension axiom}. For each $X\in \Sch_k$, let $\underline{\sZ}_*(X)$ be the quotient of $\sZ_*(X)$ by the subgroup generated by elements of the form
\[
(f:Y\to X, \pi^*(L_1),\ldots,\pi^*(L_r),M_1,\ldots, M_s)
\]
where $\pi:Y\to Z$ is a smooth morphism in $\Sm/k$, $L_1,\ldots, L_r$ are line bundles on $Z$ and $r>\dim_kZ$.
\item[b.] {\em The Gysin axiom}.  For each $X\in \Sch_k$, let $\underline{\Omega}_*(X)$ be the quotient of $\underline{\sZ}_*(X)$ by the subgroup generated by elements of the form
\[
(f:Y\to X;L_1,\ldots, L_r)-(f\circ i:Z\to X;i^*L_1,\ldots, i^*L_{r-1})
\]
where $i:Z\to Y$ is the inclusion of a smooth codimension one closed subscheme $Z$ such that $O_Y(Z)\cong L_r$.
\item[c.] {\em The formal group law}. $\Omega_*(X)$ is the quotient of $\L_*\otimes\underline{\Omega}_*$ by the $\L_*$-submodule generated by elements of the form
\[
f_*([F_A(\tilde{c}_1(L),\tilde{c}_1(M)) -\tilde{c}_1(L\otimes M)](\eta)),
\]
as $f:Y\to X$ runs over projective morphisms with $Y\in\Sm/k$ irreducible,  $L,M$ run over line bundles on $Y$, and $\eta$ runs over elements of $\sZ_*(Y)$ of the form $\tilde{c}_1(L_1)\circ\ldots\circ\tilde{c}_1(L_r)(\id_Y:Y\to Y)$ for line bundles $L_1,\ldots, L_r$ on $Y$.
\end{enumerate}
It follows from the results of \cite[\S2.4]{LevineMorel} that all the above operations are well-defined, and that $\underline{\sZ}_*$, $\underline{\Omega}_*$ and $\Omega_*$ inherit the operations of projective push-forward, smooth pull-back, Chern class operators and external products from $\sZ_*$. Finally, by \cite[Theorem 2.4.13]{LevineMorel}, $\Omega_*$ is the universal oriented Borel-Moore $\L_*$-functor of geometric type.

To extend $\vartheta^2_H$ to the desired natural transformation $\vartheta_H$, we need only show that 
$\vartheta^2_H$  sends to zero the elements described in (i)-(iii) above.  In fact, we  note that 
\begin{lem} Let $A$ be a bi-graded oriented ring cohomology theory on $\SP$. Then the restriction of $A^{2*,*}$ (with the integration on $A$ subjected to the given orientation) to $\Sm/k$ defines an oriented cohomology theory on $\Sm/k$, in the sense of \cite[Definition 1.1.2]{LevineMorel}.
\end{lem}
\begin{proof} Indeed, an oriented cohomology theory on $\Sm/k$ (following \cite{LevineMorel}) is a contravariant functor $A^*$ from $\Sm/k$ to graded, commutative rings with unit, plus push-forward maps $f_*:A^*(Y)\to A^{*+d}(X)$ for each projective morphism $f:Y\to X$, $d=\codim f$, satisfying the functoriality of projective push-forward, commutativity of pull-back and push-forward in transverse cartesian squares, the projective bundle formula (with $c_1(L):=s^*s_*(1_X)$ for $L\to X$ a line bundle with zero-section $s$) and an extended homotopy property:
\[
p^*:A^*(X)\to A^*(E)
\]
is an isomorphism for each affine space bundle $p:E\to X$. These properties for an oriented ring cohomology theory are all verified in \cite{Panin2}. 
\end{proof}

By \cite[Theorem 7.1.1]{LevineMorel} the structures we have defined on $\Omega^*$ admit a unique extension to make $\Omega^*$ an oriented cohomology theory on $\Sm/k$,  in the sense of \cite{LevineMorel}. By \cite[Theorem 7.1.3]{LevineMorel}, $\Omega^*$ is the universal oriented  cohomology theory on $\Sm/k$. Thus, given a bi-graded oriented ring cohomology theory $A$ on $\SP$, there is a unique natural transformation of oriented cohomology theories on $\Sm/k$
\[
\vartheta^A:\Omega^*\to A^{2*,*}.
\]
By \cite[Proposition 5.2.1]{LevineMorel} the Chern class operators in $\Omega^*$ are given by cup product with the Chern classes $c_1(L)$. As the image in $\Omega^*(M)$, $[f:Y\to M;L_1,\ldots, L_r]$,  of a cobordism cycle $(f:Y\to M;L_1,\ldots, L_r)$ is 
equal to $f_*(\tilde{c}_1(L)\circ\ldots\circ\tilde{c}_1(L_r)(p_Y^*(1)))$, it follows that 
\[
\vartheta^A([f:Y\to M;L_1,\ldots, L_r]])=\vartheta^A_1((f:Y\to M;L_1,\ldots, L_r))
\]
Also, $\vartheta^A:\Omega^*(pt)\to A^{2*,*}(pt)$ is a graded ring homomorphism, and 
$\vartheta^A$ is a $\Omega^*(pt)$-module homomorphism, so 
\[
\vartheta^A(a\cdot [f:Y\to M;L_1,\ldots, L_r]])=\vartheta^A_2(a\otimes (f:Y\to M;L_1,\ldots, L_r))
\]
for all $a\in\L_*$, in other words, $\vartheta^A_2$ descends to the natural transformation $\vartheta^A:\Omega^*\to A^{2*,*}$.

This immediately implies that $\vartheta^2_H$ descends to a natural transformation
\[
\vartheta_H:\Omega_*\to H_{2*,*}.
\]
Indeed, the elements described in (i)-(iii) are all of the form $f_*(\tau)$, for $\tau$ an element of $\sZ_*(Y)$, $\underline{\sZ}_*(Y)$ or $\L_*\otimes\underline{\Omega}_*(Y)$, with $Y\in \Sm/k$, $f:Y\to X$ a projective morphism, and $\tau$ going to zero in $\Omega_*(Y)$. Since 
\[
0=\vartheta^A_2(Y)(\tau)=\alpha_{Y,Y}(\vartheta^2_H(Y)(\tau)),
\]
 it follows that $\vartheta_H^2(Y)(\tau)=0$, and thus
 \[
 0=f_*(\vartheta_H^2(Y)(\tau))=\vartheta_H^2(X)(f_*(\tau)).
 \]
 Thus, $\vartheta^2_H$ descends uniquely to 
 \[
 \vartheta_H:\Omega_*\to H_{2*,*},
 \]
 completing the proof of (1)-(4).
 \end{proof}
 
 There is still the question of the behavior of $\vartheta_H$ with respect to cap products and external products. We recall \cite[Theorem 7.1.1]{LevineMorel}, which states that $\Omega_*$ admits functorial pull-back maps for all l.c.i.\! morphisms in $\Sch_k$, extending the pull-back maps for smooth morphisms, and satisfying the axioms of an oriented Borel-Moore homology theory on $\Sch_k$ (in the sense of \cite[Definition 5.1.2]{LevineMorel}). This enables us to define a cap product map
 \[
 f^*(-)\cap:\Omega^q(M)\otimes\Omega_p(Y)\to \Omega_{p-q}(Y)
 \]
 for each morphism $f:Y\to M$, $M\in\Sm/k$. Indeed, we have the external product
 \[
 \times:\Omega^q(M)\otimes\Omega_p(Y)\to \Omega_{p-q+d_M}(M\times Y)
 \]
As $M$ is smooth, the graph embedding $(f,\id_Y):Y\to M\times Y$ is a regular embedding of codimension $d_M$, so we have a well-defined pull-back
 \[
(f,\id_Y)^*:\Omega_{p-q+d_M}(M\times Y)\to \Omega_{p-q}(Y).
 \]
 We set $f^*(a)\cap b:=(f,\id_Y)^*(a\times b)$ for $a\in \Omega^q(M)$, $b\in \Omega_p(Y)$.
 
 \begin{prop}\label{prop:CapProdComp} Let $f:Y\to M$ be a morphism in $\Sch_k$, with $M\in\Sm/k$. Then the diagram
\[
\xymatrix{
\Omega^q(M)\otimes\Omega_p(Y)\ar[r]^-{f^*(-)\cap}\ar[d]_{ \alpha_{M,M}\circ\vartheta^A(M)\otimes
\vartheta_H(Y)}&\Omega_{p-q}(Y)\ar[d]^{\vartheta_H(Y)}\\
A^{2q,q}(M)\otimes H_{2p,p}(Y)\ar[r]_-{f^*(-)\cap} &H_{2(p-q),p-q}(Y)
}
\]
commutes.
\end{prop}

\begin{proof} Suppose first that $Y$ is in $\Sm/k$, of dimension $d_Y$ over $k$. It is easy to see that the map
\[
f^*(-)\cap:\Omega^q(M)\otimes \Omega_p(Y)\to \Omega_{p-q}(Y)
\]
is given by
\[
f^*(a)\cap b=f^*(a)\cup b,
\]
after making the identification $\Omega_n(Y)=\Omega^{d_Y-n}(Y)$, where the $f^*$ on the right-hand side is the pull-back map $f^*:\Omega^*(Y)\to \Omega^*(M)$ and $\cup$ is the product on 
$\Omega^*(Y)$. The analogous formula on the $(H,A)$ side follows from definition~\ref{def:OrientedDuality}(A4). Thus, the proposition is true for $Y\in\Sm/k$.

In general, we recall from \cite[Lemma 2.5.11]{LevineMorel} that $\Omega_*(Y)$ is generated (as an abelian group) by the classes of the form $[g:W\to Y]$, $W\in\Sm/k$, $g$ projective. For both $\Omega_*$ and $H_{*,*}$, we have the identity
\[
g_*((f\circ g)^*(a)\cap b)=f^*(a)\cap g_*(b).
\]
Since $\vartheta_H$ commutes with projective push-forward, and $\vartheta^A$ commutes with pull-back by arbitrary morphisms in $\Sm/k$, the case of smooth $Y$ implies the general case.
\end{proof}

\begin{prop} The natural transformation $\vartheta_H$ is compatible with external products: For $a\in\Omega_p(X)$, $b\in\Omega_q(Y)$
\[
\vartheta_H(a\times b)=\vartheta_H(a)\times\vartheta_H(b)\in H_{2(p+q),p+q}(X\times Y).
\]
\end{prop}

\begin{proof} The proof is similar to that of proposition~\ref{prop:CapProdComp}. Since the external products are compatible with push-forward   (as in definition~\ref{def:OrientedDuality}(A3)(ii)), it suffices to handle the case of smooth $X$ and $Y$. The statement is then a consequence of the fact that $\vartheta^A(M):\Omega^*(M)\to A^{2*,*}(M)$ is a ring homomorphism.
\end{proof}

\noindent
{\em Comparing $\Omega^*$ and $\MGL'_{2*,*}$}. Putting proposition~\ref{prop:MGL} and proposition~\ref{prop:OmegaUniv} together, we have the natural transformation
\[
[\vartheta_{\MGL'}:\Omega_*\to\MGL'_{2*,*}]:\Sch_k'\to \gr\Ab
\]
extending the natural transformation of oriented cohomology theories on $\Sm/k$
\[
\vartheta^{\MGL}:\Omega^*\to\MGL^{2*,*}
\]
discussed in \cite{LevineMorel}

\begin{conj} Let $k$ be a field of characteristic zero. Then 
$\vartheta_{\MGL'}:\Omega_*\to\MGL'_{2*,*}$ is an isomorphism.
\end{conj}
The analogous conjecture for $\vartheta^\MGL$ was stated in \cite{LevineMorel}. In fact, the extension of $\vartheta^\MGL$ to $\vartheta_{\MGL'}$ should allow one to use localization to prove the conjecture. We give a sketch of the argument here, details will appear in a subsequent paper.

It follows from an unpublished work of Hopkins-Morel, constructing a spectral sequence from $\L^*\otimes H^*(-,\Z(*))$ converging to $\MGL^{*,*}$,  that the map $\vartheta^\MGL(\Spec F)$ is an isomorphism for any field $F$ (in characteristic zero). Now that we have the extension to 
$\vartheta_{\MGL'}$, we can use the right-exact localization sequence and induction on the Krull dimension to prove the result in general. 

Indeed, for a given $X\in\Sch_k$, let 
\[
\Omega^{(1)}_*(X)=\lim_{\substack{\to\\W\subset X}}\Omega_*(W)
\]
where the limit is over all closed subsets of $W$ not containing any generic point of $X$. Define $\MGL_{2*,*}^{\prime(1)}(X)$ similarly. We have the commutative diagram
\[
\xymatrixcolsep{10pt}
\xymatrix{
&\Omega^{(1)}_*(X)\ar[r]^{i_*}\ar[d]_{\vartheta^{(1)}(X)}&\Omega_*(X)\ar[d]_{\vartheta(X)}\ar[r]^{j^*}&
\Omega_*(k(X))\ar[r]\ar[d]^{\vartheta(k(X))}&0\\
\MGL'_{2*+1,*}(k(X))\ar[r]_-\del&
\MGL_{2*,*}^{\prime(1)}(X)\ar[r]_-{i_*}&\MGL'_{2*,*}(X)\ar[r]_-{j^*}&\MGL_{2*,*}(k(X)\ar[r]&0
}
\]
with exact rows. Assuming $\vartheta^{(1)}(X)$ is an isomorphism, and noting that $\vartheta(k(X))$ is an isomorphism, we already find that  $\vartheta(X)$ is surjective. To show that $\vartheta(X)$ is injective, we need only lift the map $\del$ to a commutative diagram
\[
\xymatrix{
\widetilde{\MGL}'_{2*+1,*}(k(X))\ar[d]_{\vartheta'}\ar[r]^-{\del'}&\Omega^{(1)}_*(X)\ar[d]^{\vartheta^{(1)}(X)}\\
\MGL'_{2*+1,*}(k(X))\ar[r]_-\del&
\MGL_{2*,*}^{\prime(1)}(X)
}
\]
such that $i_*\circ\del'=0$ and $\vartheta'$ is surjective. For this, one uses the Hopkins-Morel spectral sequence to get a handle on elements generating $\MGL'_{2*+1,*}(k(X))$ and then   the formal group law to understand the boundary map $\del$ on the generators.

\end{document}